\documentclass[a4paper,reqno,11pt]{amsart}
\usepackage{amssymb,amscd,amsxtra,psfrag,xypic}
\usepackage[dvipdfmx]{graphicx}
\usepackage{color}

\makeatletter
\@addtoreset{equation}{section}

\makeatother

\usepackage[pdftex]{hyperref}
\hypersetup{%
bookmarksnumbered=true,%
colorlinks=true,%
setpagesize=false,%
pdftitle={},%
pdfauthor={Takuzo Okada}}

\theoremstyle{plain}
\newtheorem{Thm}{Theorem}[section]
\newtheorem{Lem}[Thm]{Lemma}
\newtheorem{Cor}[Thm]{Corollary}
\newtheorem{Prop}[Thm]{Proposition}

\theoremstyle{definition}
\newtheorem{Def}[Thm]{Definition}
\newtheorem{Def-Lem}[Thm]{Definition-Lemma}
\newtheorem{Cond}[Thm]{Condition}
\newtheorem{Rem}[Thm]{Remark}

\newtheorem*{Ack}{Acknowledgments}

\theoremstyle{remark}

\newcommand{\codim}{\operatorname{codim}}

\newcommand{\prt}{\partial}

\newcommand{\Sing}{\operatorname{Sing}}
\newcommand{\Spec}{\operatorname{Spec}}

\newcommand{\singu}{\operatorname{sing}}

\newcommand{\bsum}{\sum\nolimits}

\newcommand{\mbA}{\mathbb{A}}

\newcommand{\mbC}{\mathbb{C}}

\newcommand{\mbP}{\mathbb{P}}

\newcommand{\mbZ}{\mathbb{Z}}

\newcommand{\mcF}{\mathcal{F}}

\newcommand{\mcL}{\mathcal{L}}
\newcommand{\mcM}{\mathcal{M}}

\newcommand{\mcO}{\mathcal{O}}

\newcommand{\mcV}{\mathcal{V}}
\newcommand{\mcW}{\mathcal{W}}

\newcommand{\mcX}{\mathcal{X}}

\newcommand{\mfm}{\mathfrak{m}}

\newcommand{\msp}{\mathsf{p}}
\newcommand{\msq}{\mathsf{q}}
\newcommand{\K}{\Bbbk}

\newcommand{\inj}{\hookrightarrow}

\newcommand{\crit}{\operatorname{cr}}

\newcommand{\CH}{\operatorname{CH}}

\newcommand{\rest}{\operatorname{rest}}

\title[Smooth weighted hypersurfaces that are not stably rational]{Smooth weighted hypersurfaces that are not stably rational}
\author{Takuzo~Okada}
\address{Department of Mathematics, Faculty of Science and Engineering, Saga University, Saga 840-8502 Japan}
\email{okada@cc.saga-u.ac.jp}
\subjclass[2010]{Primary 14E08; Secondary 14J45.}
\date{}

\begin{document}

\begin{abstract}
We prove the failure of stable rationality for many smooth well formed weighted hypersurfaces of dimension at least $3$.
It is in particular proved that a very general smooth well formed Fano weighted hypersurface of index one is not stably rational.
\end{abstract}

\maketitle


\section{Introduction} \label{sec:intro}

Totaro \cite{Totaro} proved the failure of stable rationality for many hypersurfaces by developing the combination of the arguments of Voisin \cite{Voisin}, Colliot-Th\'el\`ene, Pirutka \cite{CTP} and Koll\'ar \cite{Kollar1}.
To be precise, it is proved that a very general complex hypersurface of degree $d$ in $\mbP^{n+1}$ is not stably rational if $n \ge 3$ and $d \ge 2 \lceil (n+2)/3 \rceil$.  
The aim of this article is to generalize this result to smooth weighted hypersurfaces.

In the study of stable rationality of smooth weighted hypersurfaces, main objects to be considered are Fano varieties.
Smooth Fano weighted hypersurfaces of dimension $3$ are $X_4 \subset \mbP (1^4,2)$, $X_6 \subset \mbP (1^4,3)$, $X_6 \subset \mbP (1^3,2,3)$ and hypersurfaces of degree at most $4$ in $\mbP^4$.
Here the subscript of $X$ indicates the degree of the defining equation of $X$ and, for instance, $\mbP (1^4,2) = \mbP (1,1,1,1,2)$.
The failure of stable rationality of very general $X_4 \subset \mbP (1^4,2)$, $X_6 \subset \mbP (1^4,3)$ and $X_6 \subset \mbP (1^3,2,3)$ is proved in \cite{Voisin}, \cite{Beauville} and \cite{HT}, respectively.
The failure of stable rationality of a very general terminal Fano weighted hypersurface of index $1$ (which belongs to one of the famous $95$ families) is also proved in \cite{OkorbFano}.
It is proved in \cite{Okcyclic} that a very general $4$-dimensional smooth weighted hypersurface of index $1$ (i.e.\ $I_X = 1$, see below) is not stably rational.

Besides the hypersurfaces in a projective space, one of the most familiar varieties which can be described as a weighted hypersurface are cyclic covers of $\mbP^n$.
It is proved in \cite{CTPcyclic} and \cite{Okcyclic} that a cyclic cover of $\mbP^n$ branched along a very general hypersurface of degree $d$ is not stably rational if $d \ge n+1$.

For a smooth well formed weighted hypersurface $X = X_d \subset \mbP (a_0,\dots,a_{n+1})$ of degree $d$, we define
\[
\begin{split}
a_{\Sigma} &= a_0 + a_1 + \cdots + a_{n+1}, \\
a_{\Pi} &= a_0 a_1 \cdots a_{n+1}, \\
I_X &= a_{\Sigma} - d.
\end{split}
\]
We see that $X$ is a Fano manifold if and only if $I_X > 0$, and in this case $I_X$ is called the {\it index} of $X$.
Note that $X$ is not stably rational if $I_X \le 0$.
It is known that the weights $a_0,\dots,a_{n+1}$ are mutually coprime to each other and that $d$ is divisible by $a_{\Pi}$ (see Lemma \ref{lem:charactsmoothwh}).
The following are the main results of this article.

\begin{Thm} \label{mainthm1}
Let $X$ be a very general smooth well formed weighted hypersurface of degree $d$ in $\mbP_{\mbC} (a_0,\dots,a_{n+1})$, where $n \ge 3$.
If the inequality
\[
I_X \le \max \{a_0,\dots,a_{n+1}\},
\]
holds, then $X$ is not stably rational. 
\end{Thm}

\begin{Thm} \label{mainthm2}
Let $X$ be a very general smooth well formed weighted hypersurface of degree $d$ in $\mbP_{\mbC} (a_0,\dots,a_{n+1})$, where $n \ge 3$.
Suppose that $e := d/a_{\Pi} > 1$ and let $p$ be the smallest prime factor of $e$.
If the inequality
\[
d \ge \frac{p}{p+1} a_{\Sigma}
\]
holds, then $X$ is not stably rational.
\end{Thm} 

\begin{Thm} \label{mainthm3}
Let $X$ be a very general smooth well formed weighted hypersurface of degree $d$ in $\mbP_{\mbC} (a_0,\dots,a_{n+1})$, where $n \ge 3$.
Suppose that $e := d/a_{\Pi} > 1$ is odd.
If the inequality
\[
d \ge a_{\Pi} + \frac{2}{3} a_{\Sigma}
\]
holds, then $X$ is not stably rational.
\end{Thm} 

Focusing on the index of Fano hypersurfaces, Totaro's result can be interpreted as follows: a very general Fano hypersurface of index $I$ in $\mbP^{n+1}$ is not stably rational for $n \ge 3 I$.
As a corollary to the above theorems, we can generalize this to weighted hypersurfaces.

\begin{Cor} \label{maincor1}
For a given integer $I$, there exists a constant $N_I$ depending only on $I$ such that a very general smooth well formed weighted hypersurface of dimension $n$ which is not a linear cone is not stably rational for $n \ge N_I$.
\end{Cor}

\begin{Cor} \label{maincor2}
Let $X$ be a very general smooth well formed weighted hypersurface of index $I_X$ and of dimension at least $3$, which is not a linear cone.
Then the following hold.
\begin{enumerate}
\item If $I_X = 1$, then $X$ is not stably rational.
\item If $I_X = 2$, then $X$ is not stably rational except possibly for $X_3 \subset \mbP^4$ and $X_5 \subset \mbP^6$.
\item If $I_X = 3$, then $X$ is not stably rational except possibly for $X_2 \subset \mbP^4$, $X_3 \subset \mbP^5$, $X_4 \subset \mbP^6$, $X_5 \subset \mbP^7$ and $X_7 \subset \mbP^9$.
\end{enumerate}
\end{Cor}

This implies that we can take $N_1 = 3$, $N_2 = 6$ and $N_3 = 9$ although $N_2, N_3$ may not be optimal.
In the above exceptions, $X_2 \subset \mbP^4$ is clearly rational and $X_3 \subset \mbP^4$ is not rational by \cite{CG} while its stable rationality is unknown.
Neither rationality nor stable rationality is determined for $X_5 \subset \mbP^6$ and $X_d \subset \mbP^{d+2}$ for $d = 3,4,5,7$.

We explain the content of the paper.
In Section \ref{sec:prelim}, we recall the specialization arguments of universal $\CH_0$-triviality and the Koll\'ar's construction of global differential forms on inseparable covering spaces.
In Section \ref{sec:smwhyp}, we study weighted hypersurfaces in arbitrary characteristic with an emphasis on singularities and on the restriction maps of global sections of sheaves.
Sections \ref{sec:pfthm2}, \ref{sec:pfthm3} and \ref{sec:pfthm1} are devoted to the proof of Theorems \ref{mainthm2}, \ref{mainthm3} and \ref{mainthm1}, respectively.
Theorems \ref{mainthm2} and \ref{mainthm3} can be thought of as direct generalizations of Totaro's result on hypersurfaces, while in the proof of Theorem \ref{mainthm1} we need to consider a mixed characteristic degeneration different from Totaro's.
In Section \ref{sec:supple}, we give a supplemental result on the failure of stable rationality of smooth weighted hypersurfaces and in Section \ref{sec:pfcor} we prove corollaries. 

\begin{Ack}
The author is partially supported by JSPS KAKENHI Grant Number 26800019.
\end{Ack}

\section{Preliminaries} \label{sec:prelim}

We briefly recall fundamental results which will be necessary in the proof of stable non-rationality of varieties via the reduction modulo $p$ arguments.

\subsection{Specialization of universal $\CH_0$-triviality}

For a variety $X$, we denote by $\CH_0 (X)$ the Chow group of $0$-cycles on $X$, which is by definition the free abelian group of $0$-cycles modulo rational equivalence.

\begin{Def}
\begin{enumerate}
\item A projective variety $X$ defined over a field $k$ is {\it universally $\CH_0$-trivial} if, for any field extension $F \supset k$, the degree map $\deg \colon \CH_0 (X_F) \to \mbZ$ is an isomorphism.
\item A projective morphism $\varphi \colon Y \to X$ defined over a field $k$ is {\it universally $\CH_0$-trivial} if, for any field extension $F \supset k$, the pushforward map $\varphi_* \colon \CH_0 (Y_F) \to \CH_0 (X_F)$ is an isomorphism.
\end{enumerate}
\end{Def}

Universal $\CH_0$-triviality is an obstruction for stable rationality.

\begin{Lem}
If $X$ is a smooth, projective, stably rational variety, then $X$ is universally $\CH_0$-trivial.
\end{Lem}

We apply the following form of specialization result on universal $\CH_0$-triviality.

\begin{Thm}[{\cite[Th\'eor\`eme 1.14]{CTP}}] \label{thm:sp}
Let $A$ be a discrete valuation ring with fraction field $K$ and residue field $k$, with $k$ algebraically closed.
Let $\mcX$ be a flat proper scheme over $A$ with geometrically integral fibers.
Let $X$ be the generic fiber $\mcX \times_A K$ and $Y$ the special fiber $\mcX \times_A k$.
Assume that the geometric generic fiber $X_{\overline{K}}$ is smooth, where $\overline{K}$ is an algebraic closure of $K$, and $Y$ admits a universally $\CH_0$-trivial resolution $\tilde{Y} \to Y$ of singularities.
If $X_{\overline{K}}$ is universally $\CH_0$-trivial, then so is $\tilde{Y}$.
\end{Thm}

Failure of universal $\CH_0$-triviality can be concluded by the existence of a global differential form.

\begin{Lem}[{\cite[Lemma 2.2]{Totaro}}] \label{lem:Totaro}
Let $X$ be a smooth projective variety over a filed.
If $H^0 (X,\Omega_X^i) \ne 0$ for some $i > 0$, then $X$ is not stably rational.
\end{Lem}

\subsection{Inseparable covers and global differential forms}

Let $Z$ be a smooth variety defined over an algebraically closed field $\K$ of characteristic $p > 0$, $\mcL$ an invertible sheaf on $Z$, $m$ a positive integer divisible by $p$ and $s \in H^0 (Z, \mcL^m)$.

Let $U = \Spec (\oplus_{i \ge 0} \mcL^{-i})$ be the total space of the line bundle $\mcL$ and $\pi_U \colon U \to Z$ the natural morphism.
We denote by $y \in H^0 (U, \pi_U^* \mcL)$ the zero section and define
\[
Z [\sqrt[m]{s}] = (y^m - \pi_U^* s = 0) \subset U.
\]
Set $X = Z [\sqrt[m]{s}]$ and $\pi = \pi_U|_X \colon X \to Z$.
We call $X$ or $\pi \colon X \to Z$ the {\it covering of $Z$ obtained by taking the $m$th roots of $s$}.

The singularities of $X$ can be analyzed by critical points of the section $s$.
Let $\msq \in Z$ be a point and $x_1,\dots,x_n$ local coordinates of $Z$ at $\msq$.
Around $\msq$, we can write $s = f (x_1,\dots,x_n) \tau^m$, where $f \in \mcO_{Z,\msq}$ and $\tau$ is a local generator of $\mcL$ at $\msq$.
We write $f = \alpha + \ell + q + g$, where $\alpha \in \K$, $\ell, q$ are linear, quadratic forms in $x_1,\dots,x_n$, respectively, and $g = g (x_1,\dots,x_n) \in (x_1,\dots,x_n)^3$.

\begin{Def}
We keep the above setting.
We say that $s \in H^0 (Z, \mcL^m)$ has a {\it critical point} at $\msq \in Z$ if $\ell = 0$.

We say that $s \in H^0 (Z, \mcL^m)$ has an {\it admissible critical point} at $\msq \in Z$ if $s$ has a critical point at $\msq$ and the following is satisfied:
\begin{itemize}
\item In case either $p \ne 2$ or $p = 2$ and $n$ is even, $q$ is a nondegenerate quadric.
\item In case $p = 2$, $n$ is odd and $4 \nmid m$, we have 
\[
\operatorname{length} (\mcO_{Z,\msq}/ (\prt f/\prt x_1, \dots, \prt f/\prt x_n)) = 2,
\]
or equivalently $q = \beta x_1^2 + x_2 x_3 + x_4 x_5 + \cdots + x_{n-1} x_n$ and $x_1^3 \in c$ for some $\beta \in \K$ under a suitable choice of local coordinates.
\item In case $p = 2$, $n$ is odd and $4 \mid m$, we have 
\[
\operatorname{length} (\mcO_{Z.\msq}/(\prt f/\prt x_1, \dots, \prt f/\prt x_n)) = 2
\]
and the quadric in $\mbP^{n-1}$ defined by $q = 0$ is smooth, or equivalently, $q = x_1^2 + x_2 x_3 + x_4 x_5 + \cdots + x_{n-1} x_n$ and $x_1^3 \in c$ under a suitable choice of local coordinates.
\end{itemize}
\end{Def}

Note that admissible critical points are isolated.
It is easy to see that $X$ is singular at $\msp \in X$ if and only if $s$ has a critical point at $\pi (\msp)$.
Thus, if the section $s$ has only admissible critical points on $Z$, then the singularity of $X$ are isolated.

\begin{Rem} \label{rem:admcr}
We briefly recall the argument showing that a general section $s \in H^0 (Z, \mcL^m)$ has only admissible critical points on $Z$.
We refer readers' to \cite[Section V.5]{Kollar} for details.
For a point $\msq \in Z$ and an integer $i \ge 2$, we denote by 
\[
\rest^i_{\msq} \colon H^0 (Z, \mcL^m) \to \mcL^m \otimes (\mcO_{Z,\msq}/ \mfm^i_{\msq})
\]
the restriction map, where $\mfm_{\msq} = \mfm_{Z,\msq}$ is the maximal ideal of $\mcO_{Z,\msq}$.
If $\rest^2_{\msq}$ is surjective, then the subset $V^{\crit}_{\msq} \subset H^0 (Z, \mcL^m)$ consisting of the sections admitting a critical point at $\msq$ is a linear subspace of codimension $\dim Z$ in $H^0 (Z, \mcL^m)$.
If $\rest^4_{\msq}$ is surjective, then the subset $V^{\operatorname{na}}_{\msq} \subset H^0 (Z, \mcL^m)$ consisting of the sections admitting a non-admissible critical point at $\msq$ is a proper closed subset of $V^{\crit}_{\msq}$, and hence $V^{\operatorname{na}}_{\msq}$ is of codimension at least $\dim Z + 1$ in $H^0 (Z, \mcL^m)$.
In particular, by counting dimensions, a general section in $H^0 (Z, \mcL^m)$ has only admissible critical point on $Z$ if $\rest^4_{\msq}$ is surjective for any $\msq \in Z$.
\end{Rem}

We can summarize the results of \cite{Kollar}, \cite{CTPcyclic} and \cite{Okcyclic} in the following form.

\begin{Lem}[{\cite[Chapter V.5]{Kollar}, \cite{CTPcyclic}, \cite[Proposition 4.1]{Okcyclic}}] \label{lem:covtech}
Let $X, Z, \mcL, m$ and $s$ be as above.
Assume that $s \in H^0 (Z, \mcL^m)$ has only admissible critical points on $Z$.
Then there exists an invertible subsheaf $\mcM$ of the double dual $(\Omega_X^{n-1})^{\vee \vee}$ of the sheaf $\Omega_X^{n-1}$ and a resolution of singularities $\varphi \colon \tilde{X} \to X$ with the following properties.
\begin{enumerate}
\item $\mcM \cong \pi^* (\omega_Z \otimes \mcL^m)$.
\item $\varphi$ is universally $\CH_0$-trivial and $\varphi^* \mcM \inj \Omega_{\tilde{X}}^{n-1}$
\end{enumerate}
\end{Lem}

We will refer to $\mcM$ in the above lemma as the {\it invertible subsheaf of $(\Omega_X^{n-1})^{\vee \vee}$ associated to the covering $\pi \colon X \to Z$}. 

\section{Smooth weighted hypersurfaces} \label{sec:smwhyp}

Throughout the present section, we work over an algebraically closed field $\K$ of arbitrary characteristic unless otherwise specified.
We always assume that a weighted projective space $\mbP := \mbP (a_0,\dots,a_{n+1})$ is {\it well formed}, that is, $\gcd \{a_0,\dots,\hat{a_i},\dots,a_{n+1}\} = 1$ for any $i = 0,\dots,n+1$.
Let $x_0,\dots,x_{n+1}$ be the homogeneous coordinates of degree $a_0,\dots,a_{n+1}$, respectively.
When we make explicit the ground field $\K$, we put it as a subscript $\mbP_{\K} (a_0,\dots,a_{n+1})$.
The singular locus of $\mbP$ is a union of {\it singular strata}
\[
\Pi_J = \bigcap_{i \in \{0,\dots,n+1\} \setminus J} (x_i = 0) \subset \mbP
\]
for all subset $J \subset \{0,\dots,n+1\}$ with $\gcd \{\, a_j \mid j \in J \,\} > 1$.

Let $X = X_d$ be a weighted hypersurface in $\mbP := \mbP_{\K} (a_0,\dots,a_{n+1})$ of degree $d$ and let $F (x_0,\dots,x_{n+1}) = 0$ be its defining equation.
We say that $X$ is {\it quasi-smooth} if its affine cone $C_X := (F = 0) \subset \mbA^{n+2}$ is smooth outside the origin.
We say that $X$ is {\it well formed} if $X$ does not contain any singular stratum of codimension $2$ in $\mbP$.
Note that for a quasi-smooth well formed $X$, the adjunction holds:
\[
\omega_X \cong \mcO_X (d - a_{\Sigma}).
\]

\begin{Rem} \label{rem:lincone}
Let $X$ be as above.
We say that $X$ is a {\it linear cone} if its defining equation is linear with respect to some coordinate $x_i$.
In this case $X$ is isomorphic to $\mbP (a_0,\dots,\hat{a_i},\dots,a_{n+1})$.
Clearly a general weighted hypersurface of degree $d$ in $\mbP$ is a linear cone if and only if $d = a_i$ for some $i$.
Note that under the assumption of Theorem \ref{mainthm1}, \ref{mainthm2} or \ref{mainthm3}, it is easy to verify $d > a_i$ for any $i$ so that $X$ cannot be a linear cone.
\end{Rem}

Throughout the present section, we set $\mbP = \mbP_{\K} (a_0,\dots,a_{n+1})$ and 
\[
\begin{split}
U_i &= (x_i \ne 0) \subset \mbP, \ \text{for $i = 0,\dots,n+1$}, \\
U_{i,j} &= (x_i \ne 0) \cap (x_j \ne 0) \subset \mbP, \  \text{for $0 \le i < j \le n+1$}.
\end{split}
\]
We fix notation which will be valid in the rest of the paper:
for positive integers $a_0,\dots,a_{n+1}$ and $d$, 
\[
\begin{split}
a_{\max} &= \max \{a_0,\dots,a_{n+1}\}, \\
a_{\Sigma} &= a_0 + a_1 + \cdots + a_{n+1}, \\
a_{\Pi} &= a_0 a_1 \cdots a_{n+1}, \\
r &= |\{\, i \mid a_i = 1 \,\}| - 1 \in \{-1,0,\dots,n+1\},
\end{split}
\]
and we always assume that $a_0 = a_1 = \cdots = a_r = 1$ and $a_i > 1$ for $i > r$.
Finally we define
\[
\Delta = \bigcap_{a_i = 1} (x_i = 0) = (x_0 = \cdots = x_r = 0) \subset \mbP.
\]

\subsection{Open charts of weighted projective space}

We explain descriptions of open sets $U_i$ and $U_{i,j}$ when $a_i = 1$ and $a_i$ is coprime to $a_j$, respectively.

We consider $U_i$ and assume that $a_i = 1$.
By symmetry, we may assume $i = 0$.
Then we have an isomorphism
\[
U_i \cong \Spec \K [x_1/x_0^{a_1},\dots,x_{n+1}/x_0^{a_{n+1}}].
\]
By setting $\tilde{x}_i = x_i/x_0^{a_i}$, we see that $U_i$ is isomorphic to the affine space $\mbA^{n+1}_{\tilde{x}_1,\dots,\tilde{x}_{n+1}}$ with coordinates $\tilde{x}_1,\dots,\tilde{x}_{n+1}$.
Note that the restriction of the global section $x_i \in H^0 (\mbP, \mcO_{\mbP} (a_i))$ to $U_0$ is $\tilde{x}_i$.

We consider $U_{i,j}$ and assume that $a_i$ is coprime to $a_j$.
By symmetry, we may assume $i = n, j = n+1$.
We take integers $\lambda, \mu$ such that $\lambda a_n - \mu a_{n+1} = 1$ and set $Q = x_n^{\lambda} x_{n+1}^{-\mu}$.
Then we have an isomorphism
\[
U_{i,j} \cong \Spec \K [x_0/ Q^{a_0}, x_1/ Q^{a_1}, \dots, x_{n-1}/Q^{a_{n-1}}, x_n^{a_{n+1}}/x_{n+1}^{a_n}, x_{n+1}^{a_n}/x_n^{a_{n+1}}].
\]
By setting $u = x_{n+1}^{a_n}/x_n^{a_{n+1}}$ and $\tilde{x}_i = x_i/Q^{a_i}$ for $i = 0,\dots,n-1$, we see that $U_{n,n+1}$ is isomorphic to $\mbA^n_{\tilde{x}_0,\dots,\tilde{x}_{n-1}} \times (\mbA^1_u \setminus \{o\})$.
Note that, for restrictions of global sections $x_i$, $0 \le i \le n-1$, $x_n$ and $x_{n+1}$, we have
\[
x_i|_{U_{n,n+1}} = \tilde{x}_i, \ 
x_n|_{U_{n,n+1}} = u^{\mu}, \ 
x_{n+1}|_{U_{n,n+1}} = u^{\lambda}.
\]

\subsection{Restriction maps}

The following elementary result is useful in the study of restriction maps of global sections.

\begin{Lem} \label{lem:easy}
Let $a, b$ and $N$ be positive integers and suppose that $a$ is coprime to $b$ and $N \ge (a-1)(b-1)$.
Then there exist non-negative integers $k$ and $l$ such that $N = k a + l b$.
\end{Lem}

\begin{proof}
If either $a = 1$ or $b = 1$, then the assertion is trivial.
Without loss of generality we may assume $1 < a < b$.
For an integer $i$, we set $N_i = N  - i b$.
Then $N_i \not\equiv N_j \pmod{a}$ for any $0 \le i < j < a$.
Thus there exists $l \in \{0,\dots,a-1\}$ such that $N_l \equiv 0 \pmod{a}$.
By the assumption $N \ge (a-1)(b-1)$, we have $N_0 > \cdots > N_{a-2} > 0$ and $N_{a-1} \ge -a + 1$.
The latter implies that if $l = a - 1$, then $N_l = N_{a-1}$ is non-negative.
It follows that $N_l \ge 0$ in any case and we have $N_l = k a$ for some non-negative integer $k$. 
This shows the existence of $k$ and $l$.
\end{proof}

We study restriction maps in several cases.

\begin{Lem} \label{lem:restP}
\begin{enumerate}
\item Let $i$ be such that $a_i = 1$ and let $c, l$ be positive integers such that $c \ge l a_{\max}$.
Then the restriction map
\[
\rest^{l+1}_{\msp} \colon H^0 (\mbP, \mcO_{\mbP} (c)) \to \mcO_{\mbP} (c) \otimes (\mcO_{\mbP,\msp}/\mfm^{l+1}_{\msp})
\]
is surjective for any point $\msp \in U_i$.
\item Let $i \ne j$ be such that $a_i$ is coprime to $a_j$, and $c$ an integer such that $c > (a_i - 1) (a_j - 1)$.
Then the image of the restriction map
\[
\rest^2_{\msp} \colon H^0 (\mbP, \mcO_{\mbP} (c)) \to \mcO_{\mbP} (c) \otimes (\mcO_{\mbP,\msp}/\mfm_{\msp}^2)
\]
is of dimension at least $r + 1$ for any point $\msp \in U_{i,j} \cap \Delta$.
\end{enumerate}
\end{Lem}

\begin{proof}
We prove (1).
Let $\msp \in U_i$ be a point.
Replacing coordinates, we may assume $i = 0$ and $\msp = (1\!:\!0\!:\!\cdots\!:\!0)$.
The open set $U_0$ is isomorphic to the affine space $\mbA^{n+1}$ with coordinates $\tilde{x}_1,\dots,\tilde{x}_{n+1}$ and we have $x_j|_{U_0} = \tilde{x}_j$.
For any $1 \le j_1,\dots,j_l \le n+1$ and non-negative integers $m_1,\dots,m_l$ with $0 \le m_1 + \cdots + m_l \le l$, we have 
\[
c - (m_1 a_{j_1} + \cdots + m_l a_{j_l}) \ge c - l a_{\max} \ge 0,
\] 
and the restriction of the monomial 
\[
x_{j_1}^{m_1} x_{j_2}^{m_2} \cdots x_{j_l}^{m_l} x_0^{c-(m_1 a_{j_1} + \cdots + m_l a_{j_l})} \in H^0 (\mbP, \mcO_{\mbP} (c))
\]
to $U_0$ is
\[
\tilde{x}_{j_1}^{m_1} \tilde{x}_{j_2}^{m_2} \cdots \tilde{x}_{j_l}^{m_l}.
\]
This immediately shows that $\rest^{l+1}_{\msp}$ is surjective.

We prove (2).
We may assume $(i,j) = (n, n+1)$.
Then the open set $U_{n,n+1}$ is isomorphic to $\mbA^n_{\tilde{x}_0,\dots,\tilde{x}_{n-1}} \times (\mbA_u^1 \setminus \{o\})$.
Since $c - 1 \ge (a_n - 1)(a_{n+1} - 1)$, we can take non-negative integers $\nu_n,\nu_{n+1}$ such that $c - 1 = \nu_n a_n + \nu_{n+1} a_{n+1}$ by Lemma \ref{lem:easy}.
By setting $M = x_n^{\nu_n} x_{n+1}^{\nu_{n+1}}$, we have monomials 
\[
x_0 M, x_1 M, \dots, x_r M \in H^0 (\mbP, \mcO_{\mbP} (c)),
\]
and they restrict to
\[
\tilde{x}_0 u^m, \tilde{x}_1 u^m, \cdots, \tilde{x}_r u^m,
\]
where $m$ is a suitable integer.
Since $\msp \in U_{n,n+1} \cap \Delta$, the coordinates $\tilde{x}_0,\dots,\tilde{x}_r$ can be chosen as a part of local coordinates of $\mbP$ at $\msp$ (without taking a translation) and the coordinate $u$ does not vanish at $\msp$.
Thus the section $x_i M \in H^0 (\mbP, \mcO_{\mbP} (c))$ is mapped to $\tilde{x}_i \in \mcO_{\mbP} (c) \otimes (\mcO_{\mbP, \msp} / \mfm_{\msp}^2)$ and the image of $\rest^2_{\msp}$ is of dimension at least $r+1$.
\end{proof}

\begin{Lem} \label{lem:restP2}
Suppose that $a_0,\dots,a_{n+1}$ are mutually coprime to each other and let $c$ be positive integer divisible by $a_{\Pi}$.
Then the restriction map
\[
\rest_{\msp}^2 \colon H^0 (\mbP, \mcO_{\mbP} (c)) \to \mcO_{\mbP} (c) \otimes (\mcO_{\mbP,\msp}/\mfm_{\msp}^2)
\]
is surjective for any smooth point $\msp \in \mbP$.
\end{Lem}

\begin{proof}
By the assumption, the singular locus of $\mbP$ is the set $\{\, \msp_i \mid a_i > 1 \,\}$, hence the smooth locus of $\mbP$ is covered by the $U_i$ for $i$ such that $a_i = 1$ and the $U_{i,j}$ for $i \ne j$ with $a_i, a_j > 1$.
If $\msp \in U_i$ with $a_i = 1$, then $\rest_{\msp}^2$ is surjective by Lemma \ref{lem:restP}.(1) since $c \ge a_{\Pi} \ge a_{\max}$.

Suppose that $\msp \in U_{i,j}$ with $i \ne j$ and $a_i, a_j > 1$.
Without loss of generality, we may assume $(i,j) = (n,n+1)$, i.e.\ $\msp \in U_{n,n+1}$.
As in the proof of Lemma \ref{lem:restP}, $U_{n,n+1}$ is isomorphic to $\mbA^n_{\tilde{x}_0,\dots,\tilde{x}_{n-1}} \times (\mbA_u^1 \setminus \{o\})$, where $x_j|_{U_{n,n+1}} = \tilde{x}_j$ for $j = 0,\dots,n-1$.
Let $\lambda,\mu$ be positive integers such that $\lambda a_n - \mu a_{n+1} = 1$ and set $Q = x_{n}^{\lambda} x_{n+1}^{-\mu}$.
Then we have $u = x_{n+1}^{a_n}/x_n^{a_{n+1}}$.
For each $0 \le j \le n-1$, we have 
\[
c - a_j \ge a_j a_n a_{n+1} - a_j = a_j (a_n a_{n+1} - 1) \ge a_n a_{n+1} - 1 \ge (a_n-1)(a_{n+1} - 1)
\]
and thus there exists a monomial $M_j = x_n^{\lambda_j} x_{n+1}^{\mu_j}$ of degree $c - a_j$ by Lemma \ref{lem:easy}.
Hence we have global sections 
\[
x_0 M_0, x_1 M_1, \dots, x_{n-1} M_{n-1} \in H^0 (\mbP, \mcO_{\mbP} (c))
\]
which restrict to
\[
\tilde{x}_0 u^{m_0}, \tilde{x}_1 u^{m_1}, \dots, \tilde{x}_{n-1} u^{m_{n-1}}
\]
on $U_{n,n+1}$.
Now we write $c = m a_n a_{n+1}$, where $m \ge 1$.
Then the sections
\[
x_n^{m a_{n+1}}, x_n^{(m-1) a_{n+1}} x_{n+1}^{a_n}, \dots, x_{n+1}^{m a_n} \in H^0 (\mbP, \mcO_{\mbP} (c))
\]
restricts to
\[
u^{\mu m a_{n+1}}, u^{\mu m a_{n+1} + 1}, \dots, u^{\mu m a_{n+1} + m} = u^{\lambda m a_n}.
\]
We see that the image of the sections
\[
x_0 M_0,\dots,x_{n-1} M_{n-1}, x_n^{m a_{n+1}}, x_n^{(m-1) a_{n+1}} x_{n+1}^{a_n} \in H^0 (\mbP, \mcO_{\mbP} (c))
\]
generates the $\K$-vector space $\mcO_{\mbP} (c) \otimes (\mcO_{\mbP, \msp}/\mfm_{\msp}^2)$.
This completes the proof.
\end{proof}

\begin{Rem} \label{rem:restZ}
Let $Z$ be an irreducible subvariety of $\mbP$ and let $\msp \in Z$ be a point such that both $Z$ and $\mbP$ are smooth at $\msp$.
Then the surjectivity of the restriction map
\[
\rest^{l+1}_{\msp} \colon H^0 (\mbP, \mcO_{\mbP} (c)) \to \mcO_{\mbP} (c) \otimes (\mcO_{\mbP,\msp}/\mfm_{\msp}^{l+1})
\]
implies the surjectivity of the restriction map
\[
\rest^{l+1}_{Z,\msp} \colon H^0 (Z, \mcO_Z (c)) \to \mcO_Z (c) \times (\mcO_{Z, \msp}/\mfm_{Z,\msp}^{l+1}).
\]
More generally, if the image of $\rest^{l+1}_{\msp}$ is of dimension at least $m$, then the image of $\rest^{l+1}_{Z,\msp}$ is of dimension at least $m - \codim_{\mbP} (Z)$.
\end{Rem}

\subsection{Smoothness of various weighted projective varieties}

Let $\mcF = |\mcO_{\mbP} (d)|$ be the complete linear system of weighted hypersurfaces of degree $d$ in $\mbP = \mbP_{\K} (a_0,\dots,a_{n+1})$ so that $\mcF \cong \mbP^N$, where $N = h^0 (\mbP, \mcO_{\mbP} (d)) - 1$, and let $\mcW \subset \mbP \times \mcF$, together with the second projection $\mcW \to \mcF$, be the family of such weighted hypersurfaces.
We set
\[
\mcW^{\singu} = \{\, (\msp, X) \mid \text{$X$ is singular at $\msp$} \,\} \subset \mbP \times \mcF.
\]
The image of $\mcW^{\singu}$ under the second projection $\mcW^{\singu} \to \mcF$ is the space parametrizing singular weighted hypersurfaces of degree $d$ in $\mbP$.
A component $\mcV$ of $\mcW^{\singu}$ is called {\it $\mcF$-dominating} if the restriction $\mcV \to \mcF$ of the second projection $\mbP \times \mcF \to \mcF$ to $\mcV$ is dominant.
For a component $\mcV$ of $\mcW^{\singu}$, its image via the first projection $\mbP \times \mcF \to \mbP$ is denoted by $C_{\mbP} (\mcV)$ and is called the {\it $\mbP$-center} of $\mcV$.
For a component $\mcV$ of $\mcW^{\singu}$ and a point $\msp \in \mbP$, we denote by $\mcV_{\msp}$ the fiber over $\mbP$ of the projection $\mcV \to \mbP$.
We define $\Delta_{i,j} = \Delta \cap U_{i,j} \subset \mbP$.

\begin{Lem} \label{lem:crigensm}
Suppose that $d \ge a_{\max}$.
Then the following assertions hold.
\begin{enumerate}
\item For any $\mcF$-dominating component $\mcV \subset \mcW^{\singu}$, its $\mbP$-center $C_{\mbP} (\mcV)$ is contained in $\Delta$.
\item Let $i \ne j$ be such that $i, j > r$ $($i.e.\ $a_i, a_j > 1$$)$ and $a_i$ is coprime to $a_j$.
Suppose that one of the following holds.
\begin{enumerate}
\item[(i)] $d$ is divisible by $a_{\Pi}$.
\item[(ii)] $d > (a_i - 1)(a_j - 1)$ and $2 r \ge n$.
\end{enumerate}
Then, for any $\mcF$-dominating component $\mcV \subset \mcW^{\singu}$, its $\mbP$-center $C_{\mbP} (\mcV)$ is disjoint from $\Delta_{i,j}$.
\end{enumerate}
\end{Lem}

\begin{proof}
We prove (1).
Let $\mcV \subset \mcW^{\singu}$ be an $\mcF$-dominating component.
Suppose that the $\mbP$-center $C = C_{\mbP} (\mcV)$ of $\mcV$ intersects $U_i = (x_i \ne 0) \subset \mbP$ for some $i = 0,1,\dots,r$.
Since $d \ge a_{\max}$, the restriction map
\[
\operatorname{rest}^2_{\msp} \colon H^0 (\mbP, \mcO_{\mbP} (d)) \to \mcO_{\mbP} (d) \otimes (\mcO_{\mbP, \msp}/\mfm_{\msp}^2)
\]
is surjective for any point $\msp \in U_i$ by Lemma \ref{lem:restP}.
This shows that, for any point $\msp \in C \cap U_i$, $\mcV_{\msp}$ is of codimension at least $n+2$ in $\mcF = \{\msp\} \times \mcF$.
Hence we have
\[
\dim \mcV \le \dim C + (\dim \mcF - (n+2)) < \dim \mcF
\] 
since $\dim C \le n+1$.
This is impossible since $\mcV$ is $\mcF$-dominating.
Thus $C$ is contained in $(x_i = 0)$ for $i = 0,\dots,r$, and (1) is proved.

We prove (2).
Let $\mcV \subset \mcW^{\singu}$ be an $\mcF$-dominating component.
By (1), $C = C_{\mbP} (\mcV)$ is contained in $\Delta$.
Suppose that $C \cap \Delta_{i,j} \ne \emptyset$.
We set $c = n+2$ and $c = r+1$ if we are in case (i) and (ii), respectively.
By Lemmas \ref{lem:restP2} and \ref{lem:restP}.(2), the image of the restriction map
\[
\rest^2_{\msp} \colon H^0 (\mbP, \mcO_{\mbP} (d)) \to \mcO_{\mbP} (d) \otimes (\mcO_{\mbP, \msp}/\mfm_{\msp}^2)
\]
is of dimension at least $c$ for any $\msp \in \Delta_{ij}$, and thus the codimension of $\mcV_{\msp}$ is at least $c$ in $\mcF$.
We have 
\[
\dim \mcV \le \dim C + (\dim \mcF - c) < \dim \mcF,
\]
where the last inequality clearly follows when we are in case (i) and follows since $\dim C \le \dim \Delta = n-r$ and $n \le 2 r$ when we are in case (ii).
This is impossible and (2) is proved.
\end{proof}

We set $\msp_i = (0\!:\!\cdots\!:\!1\!:\!\cdots\!:\!0) \in \mbP$, where the unique $1$ is in the $(i+1)$st position.
Then we have
\[
\Delta \setminus \left(\bigcup_{r < i < j \le n+1} \Delta_{i,j} \right) = \{\msp_{r+1}, \dots,\msp_{n+1}\}.
\]

\begin{Lem} \label{lem:crismcharp}
Suppose that $a_0,\dots,a_{n+1}$ are mutually coprime to each other.
\begin{enumerate}
\item If $d$ is divisible by $a_{\Pi}$, then a general weighted hypersurface of degree $d$ in $\mbP_{\K} (a_0,\dots,a_{n+1})$ is smooth.
\item If $d_1, d_2$ are both divisible by $a_{\Pi}$, then a general weighted complete intersection of type $(d_1, d_2)$ in $\mbP_{\K} (a_0,\dots,a_{n+1})$ is smooth.
\end{enumerate}
\end{Lem}

\begin{proof}
We prove (1).
Suppose that there exists an $\mcF$-dominating component $\mcV \subset \mcW^{\singu}$.
By Lemma \ref{lem:crigensm}, $C_{\mbP} (\mcV) = \{\msp_i\}$ for some $i \in \{r+1,\dots,n+1\}$, that is, a general member of $\mcF$ is singular at $\msp_i$.
On the other hand, since $d$ is divisible by $a_i$ for any $i$, a general member of $\mcF$ does not even pass through $\msp_i$ for any $i = r+1,\dots,n+1$.
This is a contradiction and there is no $\mcF$-dominating component of $\mcW^{\singu}$.
Therefore a general member of $\mcF$ is smooth. 

We prove (2).
Let $X_1$ be a general weighted hypersurface of degree $d_1$ in $\mbP$ which is smooth by (1).
By Lemma \ref{lem:restP2}, the restriction map
\[
H^0 (X_1, \mcO_{X_1} (d_2)) \to \mcO_{X_1} (d_2) \otimes (\mcO_{X_1,\msp}/\mfm^2_{X_1, \msp})
\]
is surjective for any point $\msp \in X_1$.
From this we can conclude (by counting dimensions) that a general member $Z \in |\mcO_{X_1} (d_2)|$, which is a general weighted hypersurface of type $(d_1,d_2)$, is smooth.
\end{proof}

\begin{Lem} \label{lem:crismcharp2}
Suppose that the following conditions are satisfied.
\begin{enumerate}
\item $a_i$ is coprime to $a_j$ for any $i < j$ except for $\{i,j\} = \{n,n+1\}$.
\item $d$ is divisible by $a_i$ for any $i$.
\item $2r \ge n$.
\end{enumerate}
Then a general weighted hypersurface of degree $d$ in $\mbP (a_0,\dots,a_{n+1})$ is smooth outside $(x_0 = \cdots = x_{n-1} = 0)$.
\end{Lem}

\begin{proof}
Let $\mcV \subset \mcW^{\singu}$ be an $\mcF$-dominating component.
We can apply Lemma \ref{lem:crigensm}.(1), since $d \ge a_{\max}$ by the assumption (2), and also apply Lemma \ref{lem:crigensm}.(2) for $r < i < j \le n+1$ with $(i,j) \ne (n,n+1)$ and conclude that $C_{\mbP} (\mcV)$ is contained in the set $\{\msp_{r+1}, \dots,\msp_{n-1}\} \cup \Gamma$, where $\Gamma = (x_0 = \cdots x_{n-1} = 0)$.
Since $d$ is divisible by $a_i$ for $i = r+1,\dots,n-1$, $C_{\mbP} (\mcV) \ne \{\msp_i\}$ for $i = r+1,\dots,n-1$.
Thus the $\mbP$-center of a $\mcF$-dominating component is contained in $\Gamma$.
Therefore a general member of $\mcF$ is smooth outside $\Gamma$.
\end{proof}

\begin{Lem} \label{lem:crismcharp3}
Suppose that the following conditions are satisfied.
\begin{enumerate}
\item $a_0,\dots,a_{n+1}$ are mutually coprime to each other.
\item There exists $k \in \{r+1,\dots,n+1\}$ such that $d$ is divisible by $a_{\Pi}/a_k$.
\item $d > (a_i-1) (a_j - 1)$ for any $0 \le i < j \le n+1$, and $d \ge a_{\max}$.
\item $2r \ge n$.
\end{enumerate}
Then a general weighted hypersurface of degree $d$ in $\mbP (a_0,\dots,a_{n+1})$ is smooth outside the point $\msp_k = (0\!:\!\cdots\!:\!1\!:\!0\!:\!\cdots\!:\!0)$.
\end{Lem}

\begin{proof}
Let $\mcV \subset \mcW^{\singu}$ be an $\mcF$-dominating component.
We can apply Lemma \ref{lem:crigensm} and conclude that $C_{\mbP} (\mcV)$ is contained in $\{\msp_{r+1}, \dots,\msp_{n+1}\}$.
Since $d$ is divisible by $a_i$ for $i \ne k$, $C_{\mbP} (\mcV) \ne \{\msp_i\}$ for $i \ne k$.
Therefore $C_{\mbP} (\mcV) \subset \{\msp_k\}$ and the proof is completed.
\end{proof}

\subsection{Characterization of smooth well formed weighted hypersurfaces}

\begin{Lem} \label{lem:charactsmoothwh}
Let $X$ be a general weighted hypersurface of degree $d$ in $\mbP_{\K} (a_0,\dots,a_{n+1})$.
Then $X$ is smooth, well formed and is not a linear cone if and only if the following conditions are satisfied:
\begin{enumerate}
\item $a_0,\dots,a_{n+1}$ are mutually coprime to each other.
\item $d$ is divisible by $a_{\Pi}$.
\item $d \ge 2 a_{\max}$.
\end{enumerate}
\end{Lem}

\begin{proof}
Set $\mbP_{\K} := \mbP_{\K} (a_0,\dots,a_{n+1})$.
Suppose that (1), (2) and (3) are satisfied.
By (3), $X$ is not a linear cone.
By (1), $\mbP_{\K}$ has at most isolated singularities and hence $X$ is clearly well formed.
By Lemma \ref{lem:crismcharp2}, $X$ is smooth.

Conversely, suppose that $X$ is smooth, well formed and is not a linear cone.
We first prove that (1), (2) and (3) hold assuming that $\operatorname{char} (\K) = 0$.
By \cite[Corollary 2.14]{PS}, $X$ is quasi-smooth, which implies $\Sing (X) = X \cap \Sing (\mbP_{\K})$.
Since $X$ is smooth, this means $X \cap \Sing (\mbP_{\K}) = \emptyset$.
Hence $\mbP_{\K}$ has at most isolated singularities and $X$ avoids those points.
The former implies (1).
The latter implies that $d$ is divisible by $a_i$ for any $i$, which implies (2).
We now know that $d$ is divisible by $a_{\max}$.
The case $d = a_{\max}$ does not happen since $X$ is not a linear cone.
Thus we have (3).
Now suppose that $\operatorname{char} (\K) = p > 0$.
Then we can lift a very general weighted hypersurface $X$ of degree $d$ in $\mbP_{\K}$ to a very general weighted hypersurface, denoted by $X_K$, of degree $d$ in $\mbP_{K} (a_0,\dots,a_{n+1})$, where $K$ is an algebraically closed field with $\operatorname{char} (K) = 0$ (using the ring of Witt vectors with residue field $\K$).
We see that $X_K$ is smooth, by the generic smoothness, and it is clearly well formed and is not a linear cone.
Then (1), (2) and (3) are satisfied by the above argument.
\end{proof}

Further restrictions are imposed on $a_0,\dots,a_{n+1}$ and $d$ when a general weighted hypersurface is in addition assumed to be Fano.

\begin{Lem} \label{lem:smwhnumerics}
Suppose that $n \ge 3$ and that a general weighted hypersurface of degree $d$ in $\mbP_{\K} (a_0,\dots,a_{n+1})$, $a_0 \le \cdots \le a_{n+1}$, is a smooth well formed Fano variety which is not a linear cone.
Then the following assertions hold.
\begin{enumerate}
\item $2 r \ge n+1$.
\item $d \ge 3 a_n$ unless $d = 2$ and $r = n+1$.
\item $d \ge 3 a_{n+1}$ unless either $d = 2 a_{n+1}$ and $r \ge n$ or $d = 2 a_{n+1}$, $r = n-1$ and $a_n = 2$.
\end{enumerate}
\end{Lem}

\begin{proof}
We note that the assumption implies that the conditions (1), (2) and (3) in Lemma \ref{lem:charactsmoothwh} are satisfied.

We prove (1).
If $r \ge n-1$, then the assertions follows immediately since $n \ge 3$.
Thus we assume $r \le n-2$.
Since $2 \le a_{r+1}, 3 \le a_{r+2}, \dots, n-r+1 \le a_n$, we have
\[
(n-r+1)! a_{n+1} \le a_{\Pi} \le d.
\]
On the other hand, the assumption that a general weighted hypersurface $X$ of degree $d$ in $\mbP (a_0,\dots,a_{n+1})$ is Fano implies
\[
d < a_{\Sigma} \le r + 1 + (n-r+1) a_{n+1}.
\]
Combining the above inequalities, we have
\[
(n-r+1) ((n-r)! - 1) a_{n+1} < r + 1.
\]
Since we are assuming $r \le n-2$, we have $(n-r)!-1 \ge 1$.
Hence we have $n-r+1 < r + 1$ and this proves (1).

We prove (2).
Suppose that $d < 3 a_n$.
Since $d$ is divisible by $a_n$ and $X$ is not a linear cone, we have $d = 2 a_n$. 
Since $a_{n+1}$ divides $d$ and $a_{n+1}$ is coprime to $a_n$, we have $a_{n+1} \le 2$.
If $a_{n+1} = 1$, then we have $r = n+1$ and $d = 2$.
Suppose that $a_{n+1} = 2$.
Then $a_0 = \cdots = a_n = 1$.
In particular we have $d = 2 a_n = 2 = a_{n+1}$.
But this is impossible since $X$ is not a linear cone.
This proves (2).

We prove (3).
Suppose that $d < 3 a_{n+1}$.
By the similar argument as above, we have $d = 2 a_{n+1}$ and $a_n \le 2$.
If $a_n = 1$, then we have $r \ge n$, and if $a_n = 2$, then $r = n-1$.
This proves (3).
\end{proof}

\section{Proof of Theorem \ref{mainthm2}} \label{sec:pfthm2}

Let $n \ge 3$, $a_0,\dots,a_{n+1}$ and $d$ be positive integers which satisfy the assumptions of Theorem \ref{mainthm2}, i.e.\ a general degree $d$ weighted hypersurface in $\mbP_{\mbC} := \mbP_{\mbC} (a_0,\dots,a_{n+1})$ is smooth and well formed, and for the smallest prime number $p$ dividing $d/a_{\Pi} > 1$, the inequality
\[
d \ge \frac{p}{p+1} a_{\Sigma}
\]
is satisfied.
Note that a very general weighted hypersurface of degree $d$ in $\mbP_{\mbC}$ is not a linear cone (cf.\ Remark \ref{rem:lincone}), hence the weights $a_i$ are mutually coprime to each other, $d \ge 2 a_{\max}$ and $d$ is divisible by $a_{\Pi}$ by Lemma \ref{lem:charactsmoothwh}.

\begin{Lem} \label{lem:known2}
\emph{Theorem \ref{mainthm2}} holds true if in addition one of the following is satisfied.
\begin{enumerate}
\item $d \ge a_{\Sigma}$.
\item $d < 2 p a_{\max}$.
\end{enumerate}
\end{Lem}

\begin{proof}
Let $W$ be a very general weighted hypersurface of degree $d$ in $\mbP_{\mbC} (a_0,\dots,a_{n+1})$.
If we are in case (1), then $H^0 (W, \omega_W) \ne 0$ and $W$ is not stably rational by Lemma \ref{lem:Totaro}.

Suppose that we are in case (2).
We assume that $a_{\max} = a_{n+1}$ and write $d = m p a_{\Pi}$ for some positive integer $m$.
Then we have
\[
d = m p a_{\Pi} < 2 p a_{\max} = 2 p a_{n+1},
\] 
which implies $m = 1$ and $a_0 = \cdots = a_n = 1$.
Then $W$ degenerates to a degree $p$ cyclic cover $W'$ of $\mbP_{\mbC}^n$ branched along a very general hypersurface of degree $d = p a_{n+1}$ and the condition of Theorem \ref{mainthm2} is equivalent to $d \ge n+1$.
Then the failure of stable rationality of $W'$, hence of $W$ by the specialization theorem \cite[Theorem 2.1]{Voisin}, is proved in \cite{CTPcyclic} and \cite{Okcyclic}.
\end{proof}

In the following we assume that we are in none of the cases (1) and (2) of Lemma \ref{lem:known2} so that $n, a_0,\dots,a_{n+1}, d$ and $p$ satisfy the following.

\begin{Cond} \label{cd2}
\begin{enumerate}
\item $a_0,\dots,a_{n+1}$ are mutually coprime to each other and $n \ge 3$.
\item $p$ is a prime number and $d$ is divisible by $p a_{\Pi}$.
\item $2 p a_{\max} \le d < a_{\Sigma}$.
\item $d \ge \frac{p}{p+1} a_{\Sigma}$.
\end{enumerate}
\end{Cond}

Note that the conditions (1), (2) and (4) follow from the assumption of Theorem \ref{mainthm2} and Lemma \ref{lem:charactsmoothwh}. 
The condition (3) is due to Lemma \ref{lem:known2}.

We explain a degeneration of weighted hypersurfaces which enables us to pass to characteristic $p$ in the proof of Theorem \ref{mainthm2}.
We set $b = d/p$.

\begin{Rem} \label{rem:spargumentThm2}
Under the above setting, we consider the variety
\[
\mcX := (y^p - f = t y - g = 0) \subset \mbP_{\mbC} (a_0,\dots,a_{n+1},b) \times \mbA^1_t,
\]
where we take $x_0,\dots,x_{n+1},y$ as homogeneous coordinates of degree $a_0,\dots,a_{n+1},b$, respectively, and $f, g \in \mbC [x_0,\dots,x_{n+1}]$ are homogeneous polynomials of degree $d$, $b$, respectively.
We assume that $f$ and $g$ are very general.
By eliminating the coordinate $y$, the fiber of $\mcX \to \mbA^1_t$ over a point except for the origin is a (very general) weighted hypersurface of degree $d$ in $\mbP_{\mbC} = \mbP_{\mbC} (a_0,\dots,a_{n+1})$ and, for the fiber $\mcX_o$ over the origin $o \in \mbA^1$, we have an isomorphism 
\[
\mcX_o \cong (y^p - f = g = 0) \subset \mbP_{\mbC} (a_0,\dots,a_{n+1},b),
\]
which is a degree $p$ cyclic cover of the weighted hypersurface $(g = 0) \subset \mbP_{\mbC}$ branched along the divisor $(f = g = 0) \subset \mbP_{\mbC}$.
This degeneration originates \cite[Example 4.3]{Mori}.
By Lemma \ref{lem:crismcharp}, both $(g = 0) \subset \mbP_{\mbC}$ and $(f = g = 0) \subset \mbP_{\mbC}$ are smooth since $\deg f = d$ and $\deg g = b = d/p$ are both divisible by $a_{\Pi}$, which implies that $\mcX_o$ is smooth.

By the degeneration theorem \cite[Theorem 2.1]{Voisin}, to prove Theorem \ref{mainthm2}, it is enough to show that $\mcX_o$ is not universally $\CH_0$-trivial.
By Theorem \ref{thm:sp}, it is then enough to work over an algebraically closed field $\K$ of characteristic $p$, consider a weighted hypersurface of the form 
\[
X :=  (y^p - f = g = 0) \subset \mbP_{\K} (a_0,\dots,a_{n+1},b),
\]
where $f, g \in \K [x_0,\dots,x_{n+1}]$ are very general, and show the existence of a universally $\CH_0$-trivial resolution $\varphi \colon \tilde{X} \to X$ such that $\tilde{X}$ is not universally $\CH_0$-trivial.
Note that $X$ is the covering of $Z := (g = 0) \subset \mbP_{\K} (a_0,\dots,a_{n+1})$ obtained by taking the $p$th roots of the section $f \in H^0 (Z, \mcO_Z (d))$.
\end{Rem}

In the following, we work over an algebraically closed field $\K$ of characteristic $p$.
Let $f, g \in \K [x_0,\dots,x_{n+1}]$ be general homogeneous polynomials of degree $d, b = d/p$, respectively.
We set
\[
\begin{split}
X &= (y^p - f = g = 0) \subset \tilde{\mbP} := \mbP_{\K} (a_0,\dots,a_{n+1},b), \\
Z &= (h = 0) \subset \mbP := \mbP_{\K} (a_0,\dots,a_{n+1}),
\end{split}
\]
and let $\pi \colon X \to Z$ be the natural morphism.
Here $x_0,\dots,x_{n+1},y$ are homogeneous of degree $a_0,\dots,a_{n+1},b$ of $\tilde{\mbP}$, respectively, and we use the same coordinates $x_0,\dots,x_{n+1}$ for the homogeneous coordinates of $\mbP$.

\begin{Lem} \label{lem:Zsm2}
$Z$ is smooth.
\end{Lem}

\begin{proof}
In view of (1) and (2) of Condition \ref{cd2}, this follows from Lemma \ref{lem:crismcharp}.
\end{proof}

We set $\mcL = \mcO_Z (b)$ and we view $f$ as an element of $H^0 (Z,\mcL^p) = H^0 (Z, \mcO_Z (d))$.
Then $\pi \colon X \to Z$ can be identified with the covering of $Z$ obtained by taking the $p$th roots of $f \in H^0 (Z, \mcL^p)$.
In the following we assume that $a_0 = \cdots = a_r = 1$ and $a_i > 1$ for any $i > r$.
We set $\Delta = (x_0 = \cdots = x_r = 0) \subset \mbP$ and $\Delta_Z = \Delta \cap Z$.

\begin{Lem} \label{lem:noncritdelta2}
A general $f \in H^0 (Z, \mcL^p)$ does not have a critical point along $\Delta_Z$.
\end{Lem}

\begin{proof}
By Conditions (1) and (2), we can apply Lemma \ref{lem:restP2} and conclude that the image of the restriction map
\[
H^0 (Z, \mcL^p) \to \mcL^p \otimes (\mcO_{Z,\msp}/\mfm_{\msp}^2)
\]
is surjective for any point $\msp \in Z$.
It follows that the sections in $H^0 (Z,\mcL^p)$ having a critical point at a given point $\msp \in \Delta_Z$ form a subspace of codimension at least $n$.
Since $\dim \Delta_Z < n$, the proof is completed by counting dimensions:
\[
\dim H^0 (Z, \mcL^p) - n + \dim \Delta_Z < \dim H^0 (Z, \mcL^p).
\]
This shows that a general $f \in H^0 (Z,\mcL^p)$ does not have a critical point along $\Delta_Z$.
\end{proof}

\begin{Lem} \label{lem:admcr2}
A general $f \in H^0 (Z,\mcL^p)$ has only admissible critical points on $Z$.
\end{Lem}

\begin{proof}
We set $U_i = (x_i \ne 0) \subset \mbP$ and $U = U_0 \cup \cdots \cup U_r$.
Note that $U = \mbP \setminus \Delta$.
Since $d \ge 2 p a_{\max} > 3 a_{\max}$ by Condition \ref{cd2}.(3), the restriction map
\[
\rest^4_{\msp} \colon H^0 (Z, \mcL^p) \to \mcL^p \otimes (\mcO_{Z,\msp}/\mfm_{\msp}^4)
\]
is surjective for any point $\msp \in Z \cap U$ by Lemma \ref{lem:restP}.(1).
By Remark \ref{rem:admcr}, a general $f \in H^0 (Z, \mcL^p)$ has only admissible critical point on $U$.
This, together with Lemma \ref{lem:noncritdelta2}, completes the proof.
\end{proof}

\begin{Prop} \label{prop:exisresol2}
The variety $X$ admits a universally trivial resolution $\varphi \colon \tilde{X} \to X$ of singularities such that $\tilde{X}$ is not universally $\CH_0$-trivial.
\end{Prop}

\begin{proof}
Let $\mcM$ be the invertible subsheaf of $(\Omega_X^{n-1})^{\vee \vee}$ associated to the covering $\pi \colon X \to Z$.
We have an isomorphism
\[
\mcM \cong \pi^* (\omega_Z \otimes \mcL^p) \cong \mcO_X \left(b - a_{\Sigma} + d \right).
\]

By Lemmas \ref{lem:admcr2} and \ref{lem:covtech}, $X$ admits a universally $\CH_0$-trivial resolution $\varphi \colon \tilde{X} \to X$ such that $\varphi^*\mcM \inj \Omega_{\tilde{X}}^{n-1}$.
We see that $H^0 (X, \mcM) \ne 0$ since
\[
b - a_{\Sigma} + d = \frac{p+1}{p} d - a_{\Sigma} \ge 0
\]
by Condition \ref{cd2}.(5).
This shows that $H^0 (\tilde{X},\Omega_{\tilde{X}}^{n-1}) \ne 0$, hence $\tilde{X}$ is not universally $\CH_0$-trivial by Lemma \ref{lem:Totaro}.
\end{proof}

Theorem \ref{mainthm2} follows from Remark \ref{rem:spargumentThm2} and Proposition \ref{prop:exisresol2}.

\section{Proof of Theorem \ref{mainthm3}} \label{sec:pfthm3}

The aim of this section is to prove Theorem \ref{mainthm3} following the Totaro's degeneration (to a reducible variety) which will be explained below.
We assume that $a_0,\dots,a_{n+1}, d$ and $e = d/a_{\Pi}$ are positive integers satisfying the assumptions of Theorem \ref{mainthm3}.

\begin{Lem} \label{lem:known3}
\emph{Theorem \ref{mainthm3}} holds true if in addition one of the following is satisfied.
\begin{enumerate}
\item $d \ge a_{\Sigma}$.
\item $e = 3$.
\end{enumerate}
\end{Lem}

\begin{proof}
Let $W = W_d \subset \mbP_{\mbC} (a_0,\dots,a_{n+1})$ be a very general weighted hypersurface of degree $d$.
If (1) is satisfied, then $W$ is clearly not stably rational.
Suppose that $e = 3$, i.e.\ $d = 3 a_{\Pi}$.
Then the condition $d \ge a_{\Pi} + \frac{2}{3} a_{\Sigma}$ is equivalent to $d \ge a_{\Sigma}$, hence $W$ is not stably rational.
\end{proof}

In the following we assume that we are not in (1) or (2) of Lemma \ref{lem:known3} so that $n, a_0,\dots,a_{n+1}, d$ and $e$ satisfy the following.

\begin{Cond} \label{cd3}
\begin{enumerate}
\item $a_0,\dots,a_{n+1}$ are mutually coprime to each other.
\item $d < a_{\Sigma}$.
\item $e = d/a_{\Pi} \ge 5$ is odd.
\item $d \ge a_{\Pi} + \frac{2}{3} a_{\Sigma}$.
\end{enumerate}
\end{Cond}

\begin{Rem} \label{rem:spargumentThm3}
Let $W$ be a very general weighted hypersurface of degree $d = e a_{\Pi}$ in $\mbP_{\mbC} := \mbP_{\mbC} (a_0,\dots,a_{n+1})$.
We can degenerate $W$ to a union of a very general weighted hypersurfaces $X$ and $G$ in $\mbP_{\mbC}$ of degree $(e-1) a_{\Pi}$ and $a_{\Pi}$, respectively.
Note that $e-1$ is even.
Let $Y$ be the degeneration of $X$ over an algebraically closed field $\K$ of characteristic $2$ obtained as in Remark \ref{rem:spargumentThm2}, which is a purely  inseparable double cover of a very general weighted hypersurface $Z$ in $\mbP_{\K} := \mbP_{\K} (a_0,\dots,a_{n+1})$ of degree $(e-1) a_{\Pi}/2$.
Let $H$ be a very general weighted hypersurface of degree $a_{\Pi}$ in $\mbP_{\K}$ to which $G$ degenerates.
We set $b = (e-1) a_{\Pi}/2$ and write
\[
\begin{split}
Y &= (y^2 + f = g = 0) \subset \tilde{\mbP}_{\K} := \mbP_{\K} (a_0,\dots,a_{n+1},b), \\
Z &= (g = 0) \subset \mbP_{\K} = \mbP_{\K} (a_0,\dots,a_{n+1}), \\
H &= (h = 0) \subset \mbP_{\K},
\end{split}
\]
where $x_0,\dots,x_{n+1}, y$ are homogeneous coordinates of degree $a_0,\dots,a_{n+1}, b$, respectively, and $f, g, h \in \K [x_0,\dots,x_{n+1}]$ are very general homogeneous polynomials of degree $2 b = (e-1)a_{\Pi}$, $b$, $a_{\Pi}$, respectively, and set $\mcL = \mcO_Z (b)$.
Let $\pi \colon Y \to Z$ be the natural morphism which is obtained by taking the roots of $f \in H^0 (Z, \mcL^2)$.
We set 
\[
\begin{split}
Y_H &:= Y \cap \pi^{-1} (H) = (y^2 + f = g = h = 0) \subset \tilde{\mbP}_{\K}, \\
Z_H &:= Z \cap H = (g = h = 0) \subset \mbP_{\K}.
\end{split}
\]
Note that $\pi|_{Y_H} \colon Y_H \to Z_H$ is the morphism obtained by taking the roots of $f|_{Z_H} \in H^0 (Z_H, (\mcL|_{Z_H})^{2})$.

Now suppose that the following are satisfied.
\begin{enumerate}
\item $Z$ is smooth, the section $f \in H^0 (Z, \mcL^2)$ does not have a critical point on $Y_H \subset Y$ and has only admissible critical points on $Z$.
\item $Z_H$ is smooth and the section $f|_{Z_H} \in H^0 (Z_H, (\mcL|_{Y_H})^2)$ has only admissible critical points on $Z_H$.
\item The invertible subsheaf $\mcM$ of $(\Omega_Y^{n-1})^{\vee \vee}$ associated to $\pi \colon Y \to Z$ has a non-zero global section.
\item $H^0 (Z_H, \omega_{Z_H}) = 0$.
\end{enumerate}
Let $\varphi \colon \tilde{Y} \to Y$ be the universally $\CH_0$-trivial resolution of $Y$ as in Lemma \ref{lem:covtech} (note that $\varphi$ is obtained by  blowing-up each singular point of $Y$).
Here the existence of $\varphi$ follows from (1).
Replacing $\tilde{Y}$ by a further blowing up (at each singular point of $Z_H$ which is contained in the smooth locus of $Y$ by (2)) model , we may assume that the restriction $\varphi|_{\tilde{Z}_H} \colon \tilde{Z}_H \to Z_H$, where $\tilde{Z}_H$ is the proper transform of $Z_H$ by $\varphi$, is the universally $\CH_0$-trivial resolution of $Z_H$. 
Under the above assumptions, it follows from the argument in \cite{Totaro} in pages 887 and 888 that universal $\CH_0$-triviality of $W$ implies that the restriction
\[
H^0 (\tilde{Y}, \Omega_{\tilde{Y}}^{n-1}) \to H^0 (\tilde{Z}_H, \Omega_{\tilde{Z}_H}^{n-1})
\]
is injective.
By (1) and (3), we have $0 \ne H^0 (\tilde{Y}, \varphi^* \mcM) \inj H^0 (\tilde{Y}, \Omega_{\tilde{Y}}^{n-1})$.
By (4), $H^0 (\tilde{Z}_H, \Omega_{\tilde{Z}_H}^{n-1}) = H^0 (Z_H, \omega_{Z_H}) = 0$.
This is a contradiction.
Therefore, for the proof of Theorem \ref{mainthm3}, it is enough to show that (1), (2), (3) and (4) are satisfied.
\end{Rem}

We keep the same notation and setting as in Remark \ref{rem:spargumentThm3}.

\begin{Lem} \label{lem:verif1thm3}
$Z$ is smooth, the section $f \in H^0 (Z, \mcL^{2})$ does not have a critical point on $Z_H \subset Z$ and has only admissible critical points on $Z$. 
\end{Lem}

\begin{proof}
Recall that the weighted hypersurface $Z \subset \mbP$ is of degree $b = (e-1)a_{\Pi}/2$ and $b$ is divisible by $a_{\Pi}$.
By Lemma \ref{lem:crismcharp}, $Z$ is smooth.
Recall also that $\mcL^2 = \mcO_Z ((e-1) a_{\Pi})$ and $e-1 \ge 4$.
By the same argument as in the proof of Lemma \ref{lem:noncritdelta2}, we conclude that a general $f \in H^0 (Z, \mcL^2)$ does not have a critical point along a given proper closed subvariety of $Z$.
Thus $f$ does not have a critical point along $\Delta_Z \cup Z_H$.
Then, since $(e-1) a_{\Pi} \ge 3 a_{\max}$, we can apply Lemma \ref{lem:restP}.(1) and conclude that $f$ has only admissible critical points on $Z$ (by the same argument as in the proof of Lemma \ref{lem:admcr2}).
\end{proof}

\begin{Lem} \label{lem:verif2thm3}
$Z_H$ is smooth and the section $f|_{Z_H} \in H^0 (Z_H, (\mcL|_{Z_H})^2)$ has only admissible critical points on $Z_H$.
\end{Lem}

\begin{proof}
The variety $Z_H$ is a general weighted complete intersection of type $(b, a_{\Pi})$ in $\mbP_{\K}$.
By Lemma \ref{lem:crismcharp}, $Z_H$ is smooth.
We have $(\mcL|_{Z_H})^2 \cong \mcO_{Z_H} (2b)$ and $2 b = (e-1) a_{\Pi} \ge 3 a_{\max}$.
Thus we can apply Lemma \ref{lem:restP}.(1) and conclude that $f|_{Z_H} \in H^0 (Z_H, (\mcL|_{Z_H})^2)$ has only admissible critical points on $Z_H$.
\end{proof}

\begin{proof}[Proof of \emph{Theorem \ref{mainthm3}}]
It is enough to show that the conditions (1), (2), (3) and (4) in Remark \ref{rem:spargumentThm3} are satisfied.
Conditions (1) and (2) are already verified in Lemmas \ref{lem:verif1thm3} and \ref{lem:verif2thm3}.
For the invertible sheaf $\mcM \subset (\Omega_Y^{n-1})^{\vee \vee}$, we have an isomorphism
\[
\mcM \cong \pi^* (\omega_Z \otimes \mcL^2) \cong \mcO_Y \left( \frac{3}{2} d - \frac{3}{2} a_{\Pi} - a_{\Sigma} \right).
\]
We have $H^0 (Y, \mcM) \ne 0$ since $d \ge a_{\Pi} + \frac{2}{3} a_{\Sigma}$, and the condition (3) is verified.
Finally, we have
\[
\omega_{Z_H} = \mcO_{Z_H} \left( \frac{e+1}{2} a_{\Pi} - a_{\Sigma} \right).
\]
Since $(e+1)/2 < e$ and $d - a_{\Sigma} < 0$, we have
\[
\frac{e+1}{2} a_{\Pi} - a_{\Sigma} < d - a_{\Sigma} < 0,
\]
and thus the condition (4) is verified.
This completes the proof.
\end{proof}

\section{Proof of Theorem \ref{mainthm1}} \label{sec:pfthm1}

The aim of this section is to prove Theorem \ref{mainthm1}.
Let $n \ge 3$, $a_0,\dots,a_{n+1}$ and $d$ be as in Theorem \ref{mainthm1}.
Since a general weighted hypersurface of degree $d$ in $\mbP_{\mbC} (a_0,\dots,a_{n+1})$ is smooth, well formed and is not a linear cone (see Remark \ref{rem:lincone}), the weights $a_i$ are mutually coprime to each other, $d \ge 2 a_{\max}$ and $d$ is divisible by $a_{\Pi}$.

\begin{Lem} \label{lem:known1}
\emph{Theorem \ref{mainthm1}} holds true if in addition one of the following is satisfied.
\begin{enumerate}
\item $d \ge a_{\Sigma}$.
\item $r \ge n$.
\item $d \ne a_{\Pi}$.
\end{enumerate}
\end{Lem}

\begin{proof}
Let $W$ be a very general weighted hypersurface of degree $d$ in $\mbP_{\mbC} (a_0,\dots,a_{n+1})$.

If we are in case (1), then $W$ is not stably rational since $H^0 (W, \omega_W) \ne 0$.

Suppose that we are in case (2).
If $r = n+1$, i.e.\ $W$ is a hypersurface of degree $d$ in $\mbP^{n+1}$, then the stable non-rationality of $W$ follows from \cite[Theorem 2.1]{Totaro} since the condition $I_W \le a_{\max} = 1$ is equivalent to $d \ge n+1$ which is stronger than $d \ge 2 \lceil (n+2)/3 \rceil$.
If $r = n$, then $W$ can be degenerated to a (degree $d/a_{\max}$) cyclic cover of $\mbP^n_{\mbC}$ branched along a hypersurface of degree $d$.
Then stable non-rationality of $W$ follows from \cite[Theorem 1.1]{Okcyclic} since the condition $I_W \le a_{\max}$ is equivalent to $d \ge n+1$.

Suppose that we are in case (3).
By (1) and (2), we may assume that $d < a_{\Sigma}$ and $r \le n-1$.
We may assume $a_{\max} = a_{n+1}$.
Note that we have $a_0 \cdots a_n \ge 2$.
Let $p$ be the smallest prime number dividing $d/a_{\Pi}$.
Then we have
\[
a_{\Sigma} > d \ge p a_0 \cdots a_n a_{n+1} \ge 2 p a_{n+1}.
\]
By the assumption of Theorem \ref{mainthm1}, we have $I_W \le a_{n+1}$ which is equivalent to $d \ge a_{\Sigma} - a_{n+1}$.
We have
\[
\begin{split}
d  - \frac{p}{p+1} a_{\Sigma} &\ge (a_{\Sigma}-a_{n+1}) - \frac{p}{p+1} a_{\Sigma} = \frac{1}{p+1} (a_{\Sigma} - (p+1) a_{n+1}) \\
&> \frac{p-1}{p+1} a_{n+1} > 0.
\end{split}
\] 
Thus the assumption of Theorem \ref{mainthm2} is satisfied and $W$ is not stably rational.
This completes the proof. 
\end{proof}

In the following we assume that we are in none of the cases (1), (2) and (3) of Lemma \ref{lem:known1} so that $n, a_0,\dots,a_{n+1}$ and $d$ satisfy the following after re-ordering the $a_i$.

\begin{Cond} \label{cd1}
\begin{enumerate}
\item $a_0,\dots,a_{n+1}$ are mutually coprime to each other and $n \ge 3$.
\item $a_{n+1} = \max \{a_0,\dots,a_{n+1}\}$ and $a_n = \max \{a_0,\dots,a_n\} \ge 2$.
\item $3 a_n \le d = a_{\Pi} < a_{\Sigma}$.
\item $2 r > n$.
\item $d \ge \sum_{i=0}^n a_i$.
\end{enumerate}
\end{Cond}

Note that the inequality $3 a_n \le d$ follows from Lemma \ref{lem:smwhnumerics}.(2).
We choose and fix a prime number $p$ which divides $a_n$.
We set $b := a_n a_{n+1}$ and $e := d/b = a_0 \cdots a_{n-1}$.
Note that $d = e b$.

\begin{Rem} \label{rem:spargument1}
A very general smooth well formed weighted hypersurface $W$ of degree $d$ in $\mbP_{\mbC} (a_0,\dots,a_{n+1})$ degenerates to a weighted hypersurface $W'$ in $\mbP_{\mbC} (a_0,\dots,a_{n+1})$ of degree $d$ defined by an equation of the form
\[
x_{n+1}^{e a_n} + x_{n+1}^{(e-1) a_n} f_b + \cdots x_{n+1}^{a_n} f_{(e-1)b} + f_{e b} = 0,
\]
where $f_i \in \mbC [x_0,\dots,x_n]$ is a very general homogeneous polynomial of degree $i$.
The variety $W'$ is a (very) general member of a bese point free linear system, hence $W'$ is smooth by Bertini theorem (cf.\ \cite[Corollary 10.9, Remark 10.9.2]{H}).
By the degeneration theorem \cite[Theorem 2.1]{Voisin}, to prove Theorem \ref{mainthm1}, it is enough to show that $W'$ is not universally $\CH_0$-trivial.
By the specialization theorem \cite[Th\'eor\`eme 1.14]{CTP}, it is then enough to show that a weighted hypersurface $X$ in $\mbP_{\K} (a_0,\dots,a_{n+1})$ defined by an equation of the form
\[
x_{n+1}^{e a_n} + x_{n+1}^{(e-1) a_n} f_b + \cdots x_{n+1}^{a_n} f_{(e-1)b} + f_{e b} = 0,
\]
where $f_i \in \K [x_0,\dots,x_n]$ is a very general homogeneous polynomial of degree $i$, admits a universally $\CH_0$-trivial resolution $\varphi \colon \tilde{X} \to X$ such that $\tilde{X}$ is not universally $\CH_0$-trivial.
\end{Rem}

In the following we work over an algebraically closed field $\K$ of characteristic $p$ unless otherwise specified and let $\tilde{\mbP} := \mbP_{\K} (a_0,\dots,a_{n+1})$ be the weighted projective space with homogeneous coordinates $x_0,\dots,x_{n+1}$ with $\deg x_i = a_i$.
The coordinate $x_{n+1}$ will be distinguished and we denote it as $y = x_{n+1}$.
We define $X$ to be the weighted hypersurface of degree $d$ in $\tilde{\mbP}$ defined by the equation
\[
F := y^{e a_n} + y^{(e-1)a_n} f_b + y^{(e-2) a_n} f_{2b} + \cdots + y^{a_n} f_{(e - 1) b} + f_{eb} = 0,
\]
where $f_i \in \K [x_0,\dots,x_n]$ is a general homogeneous polynomial of degree $i$.
Let 
\[
\mbP := \mbP_{\K} (a_0,\dots,a_n, b)
\] 
be the weighted projective space with homogeneous coordinates $x_0,\dots,x_n$ and $z$ with $\deg x_i = a_i$ and $\deg z = b$, and let $Z$ be the weighted hypersurface of degree $d$ in $\mbP$ defined by
\[
G := z^e + z^{e-1} f_b + z^{e-2} f_{2 b} + \cdots + z f_{(e-1)b} + f_{e b} = 0.
\]
Note that $Z$ is a general weighted hypersurface of degree $d$ in $\mbP$.
The restriction of the natural morphism
\[
\tilde{\mbP} \to \mbP, 
\quad (x_0\!:\!\cdots\!:\!x_n\!:\!y) \mapsto (x_0\!:\!\cdots\!:\!x_n\!:\!y^{a_n})
\]
to $X$ is denoted by $\pi \colon X \to Z$.
We set 
\[
\begin{split}
\Gamma_Z &:= (x_0 = \cdots = x_{n-1} = 0) \cap Z \subset \mbP, \\
Z^{\circ} &:= Z \setminus \Gamma_Z, \\
\Gamma_X &:= \pi^{-1} (\Xi) = (x_0 = \cdots = x_{n-1} = 0) \cap X, \\
X^{\circ} &:= \pi^{-1} (Z^{\circ}) = X \setminus \Gamma_X.
\end{split}
\]

\begin{Lem} \label{lem:smZ1}
$Z^{\circ}$ is smooth.
\end{Lem}

\begin{proof}
This follows from Lemma \ref{lem:crismcharp2}.
\end{proof}

\begin{Lem} \label{lem:smalongGamma}
$X$ is smooth along $\Gamma_X$.
\end{Lem}

\begin{proof}
Set $V = (x_n \ne 0) \cap (y \ne 0) \subset \tilde{\mbP}_{\K}$.
We have $\Gamma_X = \Gamma_X \cap V$ since $\msp_n, \msp_{n+1} \notin X$.
Take positive integers $\lambda,\mu$ such that $\lambda a_n - \mu a_{n+1} = 1$ and set $Q = x_n^{\lambda} y^{-\mu}$.
Note that $p \nmid \mu$ since $p \mid a_n$.
Then $V$ can be identified with $\mbA^n_{\tilde{x}_0,\dots,\tilde{x}_{n-1}} \times (\mbA^1_u \setminus \{o\})$, where $\tilde{x}_i = x_i/Q^{a_i}$ and $u = y^{a_n}/x_n^{a_{n+1}}$.
We have $y|_V = u^{\lambda}$ and $x_n|_V = u^{\mu}$ so that global sections 
\[
y^{e a_n}, y^{(e-1) a_n} x_n^{a_{n+1}}, \cdots, x_n^{e a_{n+1}},
\]
restrict to functions
\[
u^{\mu e a_n}, u^{\mu e a_n - 1}, \cdots, u^{\mu e a_n - e} = u^{\lambda e a_{n+1}},
\]
on $V$.
We write 
\[
F = y^{e a_n} + \alpha_1 y^{(e-1) a_n} x_n^{a_{n+1}} + \alpha_2 y^{(e-2) a_n} x_n^{2 a_{n+1}} + \cdots + \alpha_e x_n^{e a_{n+1}} + h,
\]
where $\alpha_i \in \K$ is the coefficient of $x_n^{i a_{n+1}}$ in $f_{i b}$ and $h = h (x_0,\dots,x_n,y)$ is the remaining terms.
Note that $h \in (x_0,\dots,x_{n-1})$.
Then $X \cap V$ is defined by the equation
\[
F|_V = u^{\lambda e a_n} + \alpha_1 u^{\lambda e a_n - 1} + \alpha_2 u^{\lambda e a_n - 2} + \cdots + \alpha_e u^{\lambda e a_n - e} + \tilde{h},
\]
where $\tilde{h} = \tilde{h} (\tilde{x}_0,\dots, \tilde{x}_{n-1},u) = h|_V$.
Note that $\lambda e a_n - e = \mu e a_{n+1}$ is not divisible by $p$, $\alpha_1, \dots,\alpha_e$ are general,  $\tilde{h} \in (\tilde{x}_0,\dots, \tilde{x}_{n-1})$ and $\Gamma_X \cap V$ is defined by $\tilde{x}_0 = \cdots = \tilde{x}_{n-1} = 0$.
It is then easy to check that $X$ is smooth along $\Gamma_X$ and the proof is completed.
\end{proof}

We set $\mcL = \mcO_{Z^{\circ}} (a_{n+1})$.
We can view $z$ (or more precisely $z|_{Z^{\circ}}$) as an element of $H^0 (Z^{\circ}, \mcL^{a_n}) = H^0 (Z^{\circ}, \mcO_Z (b))$, and $\pi^{\circ} = \pi|_{X^{\circ}} \colon X^{\circ} \to Z^{\circ}$ is the covering obtained by taking the $a_n$th roots of $z$.
We define $\Delta_Z = (x_0 = \cdots = x_r = 0) \cap Z$ and $\Delta^{\circ}_Z = \Delta_Z \setminus \Gamma_Z$.

\begin{Lem} \label{lem:critDelta1}
The section $z \in H^0 (Z^{\circ}, \mcL^{a_n})$ does not have a critical point along $\Delta^{\circ}_Z$.
\end{Lem}

\begin{proof}
The section $z$ has a critical point at $\msp \in Z^{\circ}$ if and only if $X$ is singular at any point of $\pi^{-1} (\msp)$.
Thus it is enough to show that $X$ is smooth along $\Delta_X^{\circ}$, where
\[
\Delta_X^{\circ} = \pi^{-1} (\Delta_Z^{\circ}) = (x_0 = \cdots = x_r = 0) \cap X \setminus \Gamma_X.
\]

We set $\Delta_{\tilde{\mbP}} = (x_0 = \cdots = x_r = 0) \subset \tilde{\mbP}$ and $\Delta^{\circ}_{\tilde{\mbP}} = \Delta_{\tilde{\mbP}} \setminus (x_0 = \cdots = x_{n-1} = 0)$.
Let $W$ be the $\K$-vector subspace of $H^0 (\tilde{\mbP}, \mcO_{\tilde{\mbP}} (d))$ generated by the monomials
\[
\{\, x_0^{k_0} \cdots x_n^{k_n} y^{l a_n} \mid k_i, l \ge 0, \sum k_i a_i + l a_n a_{n+1} = d \,\}.
\]
Note that $X$ is defined by a general element of $W$.
We claim that the image of the restriction map
\[
\rest^2_{\msp} \colon W \to \mcO_{\tilde{\mbP}} (d) \otimes (\mcO_{\tilde{\mbP}, \msp}/\mfm^2_{\msp})
\]
is of dimension at least $r+1$ for any point $\msp \in \Delta^{\circ}_{\tilde{\mbP}}$.
We define
\begin{align*}
V_{i,j} &= (x_i \ne 0) \cap (x_j \ne 0) \subset \tilde{\mbP}, & \text{for $r < i < j \le n$}, \\
V_{i,y} &= (x_i \ne 0) \cap (y \ne 0) \subset \tilde{\mbP}, & \text{for $r < i \le n-1$}.
\end{align*}
We set $\msp_i = (0\!:\!\cdots\!:\!1\!:\!\cdots\!:\!0) \in \tilde{\mbP}$, where the unique $1$ is in the $(i+1)$st position.
Then we have
\[
\Delta^{\circ}_{\tilde{\mbP}} \setminus \{\msp_{r+1},\dots,\msp_{n-2}\} = \left( \bigcup_{r < i < j \le n} \Delta_{\tilde{\mbP}} \cap V_{i,j} \right) \cup \left( \bigcup_{r < i \le n-1} \Delta_{\tilde{\mbP}} \cap V_{i,y} \right).
\]

Suppose that $\msp \in \Delta \cap V_{i,j}$ for some $r < i \ne j \le n$.
Then, since $a_i$ is coprime to $a_j$ and $d \ge a_i a_j > (a_i-1)(a_j-1)$, there exists monomial $M = x_i^{\lambda} x_j^{\mu}$ of degree $d-1$ by Lemma \ref{lem:easy}.
The section $x_i M \in W$, for $i = 0,\dots,r$, restricts to the function $\tilde{x}_i u^l$ on $V_{i,j}$, where $l$ is a suitable integer, and $\tilde{x}_0,\dots,\tilde{x}_r$ form a part of local coordinates of $\tilde{\mbP}$ at $\msp$.
Thus the image of $\rest^2_{\msp}$ is of dimension at least $r+1$. 

Suppose that $\msp \in \Delta \cap V_{i,y}$ for some $r < i \le n-1$.
Then, since $a_i$ is coprime to $a_n a_{n+1}$ and $d \ge a_i a_n a_{n+1} > (a_i-1)(a_n a_{n+1} - 1)$, there exists a monomial $M' = x_i^{\lambda'} y^{\mu' a_n}$ of degree $d-1$.
Since $x_i M' \in W$, $i = 0,\dots,r$, we can repeat the above arguments and conclude that the image of $\rest^2_{\msp}$ is of dimension at least $r+1$.
Thus the claim is proved.

For $\msp \in \Delta^{\circ}_{\tilde{\mbP}}$, let $W_{\msp}$ be the subspace of $W$ consisting of the polynomials $H \in W$ such that the weighted hypersurface in $\tilde{\mbP}$ defined by $H = 0$ is singular at $\msp$.
By the above claim, the codimension of $W_{\msp}$ in $W$ is at least $r+1$ for any $\msp \in \Delta^{\circ} \setminus \{\msp_{r+1},\dots,\msp_{n-2}\}$. 
Since $\dim \Delta_{\tilde{\mbP}} = n-r$ and $2 r > n$ by Condition \ref{cd1}.(4), we have
\[
\dim W - (r+1) + \dim \Delta < \dim W.
\]
This shows that $X$ is smooth along $\Delta^{\circ}_{\tilde{\mbP}} \setminus \{\msp_{r+1},\dots,\msp_{n-2}\}$.
It is clear that a general $H \in W$ does not vanish at $\msp_i$ (for $i = r+1,\dots,n-2$). 
Thus $X$ is smooth along $\Delta^{\circ}_{\tilde{\mbP}}$.
\end{proof}

\begin{Lem} \label{lem:admcr1}
The section $z \in H^0 (Z^{\circ}, \mcL^{a_n})$ has only admissible critical points on $Z^{\circ}$.
\end{Lem}

\begin{proof}
We choose and fix general $f_b, f_{2b},\dots,f_{(e-1)b} \in \K [x_0,\dots,x_n]$ and we will show that the section $z \in H^0 (Z^{\circ},\mcL^{a_n})$ has only admissible critical points on $Z^{\circ}$ for a general choice of $f_d = f_{e b} \in \K [x_0,\dots,x_n]$.
Note that $Z$ itself varies as we vary $f_d$.

Let $\mcF$ be the affine space parameterizing homogeneous polynomials of degree $d = e b$ in variables $x_0,\dots,x_n$.
For a homogeneous polynomial $f_d$ of degree $d$, we denote by $[f_d] \in \mcF$ the corresponding point.
We set
\[
\mcW^{\operatorname{na}} := \{\, (\msp, [f_d]) \in \mbP \times \mcF \mid \text{$\msp \in Z^{\circ}$ and $z$ has a non-admissible critical point at $\msp$} \,\}.
\] 
It is enough to show that there is no $\mcF$-dominating component of $\mcW^{\operatorname{na}}$.
Assume to the contrary that there exists such a component $\mcV$ of $\mcW^{\operatorname{na}}$ and let $C$ be the $\mbP$-center of $\mcV$, i.e.\ the image of $\mcV$ under the first projection $\mcW^{\operatorname{na}} \to \mbP$.

For $0 \le i \le r$, we set $U_i = (x_i \ne 0) \subset \mbP$ and $U = U_0 \cup \cdots \cup U_r \subset \mbP$.
Assume that $C \cap U \ne \emptyset$.
We compute the number of independent conditions imposed for $z$ (and for $f_d$) to have a non-admissible critical point at $\msp$.
To do so we may assume $\msp = (1\!:\!0\!:\!\cdots\!:\!0\!:\!\zeta) \in U_0$ for some $\zeta \in \K$ by considering a suitable automorphism of $\mbP^{\circ}$ (which leaves $z$ invariant).
Note that, by Lemma \ref{lem:restP}.(1) and Condition \ref{cd1}.(3), the restriction map
\[
\rest^3_{\msp} \colon H^0 (\mbP, \mcO_{\mbP} (d)) \to \mcO_{\mbP} (d) \otimes (\mcO_{\mbP}/\mfm_{\msp}^4)
\]
is surjective.
For a homogeneous polynomial $h = h (x_0,\dots,x_n,z) \in H^0 (\mbP, \mcO_{\mbP} (d))$, we set $\tilde{h} = h (1,\tilde{x}_1,\dots,\tilde{x}_n,\tilde{z})$, so that $\tilde{h}$ is the restriction of the section $h$ to $U_0 \cong \mbA^{n+1}_{\tilde{x}_1,\dots,\tilde{x}_n,\tilde{z}}$.
We write 
\[
f_{i b} = \alpha_i x_0^{i b} + \ell_i x_0^{i b -1} + q_i x_0^{i b -2} + c_i x_0^{i b -3} + g_{i b},
\]
where $\ell_i, q_i, c_i$ are linear, quadratic, cubic forms in $x_1,\dots,x_n$ and $g_{ib}$ is contained in the ideal $(x_1,\dots,x_n)^4 \subset \K [x_0,\dots,x_n]$.
Note that $Z \cap U_0$ is the hypersurface in $U_0$ defined by the equation
\[
\tilde{G} = \tilde{z}^e + \tilde{z}^{e-1} \tilde{f}_b + \tilde{z}^{e-2} \tilde{f}_{2 b} + \cdots + \tilde{f}_{eb} = 0.
\]
We set
\[
\xi := \frac{\prt \tilde{G}}{\prt \tilde{z}} (\msp) 
= e \zeta^{e - 1} + \sum_{j=1}^{e-1} (e-j) \alpha_j \zeta^{(e-j) - 1} \in \K.
\]
We may assume $\xi \ne 0$ because otherwise $\tilde{z}$ (or more precisely, its translation $\tilde{z} - \zeta$) becomes a part of local coordinates of $Z^{\circ}$ at $\msp$ and $z$ does not have a critical point at $\msp$.
Then we can choose $\tilde{x}_1,\dots,\tilde{x}_n$ as local coordinates of $Z^{\circ}$ at $\msp$ and we express $\tilde{z}$ as
\[
\tilde{z} = \zeta + \ell + q + c + \cdots,
\]
where $\ell, q$ and $c$ are linear, quadric and cubic forms in variables $\tilde{x}_1,\dots,\tilde{x}_n$, respectively.

By substituting $\tilde{z} = \zeta + \ell + \cdots$ into the defining equation $\tilde{G} = 0$ of $Z \cap U_0$, we have
\[
g := \tilde{G} (\tilde{x}_1,\dots,\tilde{x}_n,\zeta + \ell + q + c + \cdots) = 0.
\]
Looking at the constant term of $g$, we have
\[
\zeta^e + \sum_{j=1}^{e-1} \alpha_i \zeta^{e-j} + \alpha_e = 0,
\]
which imposes $1$ condition on $f_d = f_{e b} = \alpha_e x_0^{eb} + \ell_e x_0^{e b -1} + \cdots$ since $\rest^3_{\msp}$ is surjective.
Looking at the linear term of $g$, we have
\[
\xi \ell + \sum_{j=1}^{e-1} \zeta^{e-j} \ell_j + \ell_e = 0,
\]
We see that $z$ has a critical point at $\msp$ if and only if $\ell = 0$ as a polynomial, which is equivalent to
\[
\ell_e = - \sum_{j=1}^{e-1} \zeta^{e-j} \ell_j.
\]
Since $\rest_{\msp}^3$ is surjective, this imposes $n$ independent conditions on $f_d$.
From now on we assume that $\ell = 0$.
Then, by looking at the quadratic and cubic terms of $g = 0$, we have
\[
\begin{split}
- \xi q &=  \sum_{j = 1}^{e-1} \zeta^{e-j} q_j + q_e, \\
- \xi c &= \sum_{j=1}^{e-1} (e-j) \zeta^{e-j-1} q \ell_i + \sum_{j=1}^{e-1} \zeta^{e-j} c_j + c_e.
\end{split}
\]
It is now easy to see that, in view of the fact that $\rest_{\msp}^3$ is surjective, $\tilde{z} = \zeta + q + c + \cdots$ has an admissible critical point at $\msp$ for a general choice of $q_e, c_e$.
This shows that the fiber $\mcW^{\operatorname{na}}_{\msp}$ of $\mcW^{\operatorname{na}} \to \mbP$ over $\msp \in U = U_0 \cup \cdots \cup U_r$ is of dimension $\dim \mcF - (n+2)$.
Since $\dim \mbP = n+1$, it follows that the $\mbP$-center $C$ of $\mcV$ is disjoint from $U$ and thus contained in $(x_0 = \cdots = x_r = 0) \subset \mbP$, that is, $z$ has only admissible critical points on $Z^{\circ} \setminus \Delta^{\circ}_Z$ for a general choice of $f_d$.
Now the proof is completed by Lemma \ref{lem:critDelta1}.
\end{proof}

\begin{Prop}
The variety $X$ admits a universally $\CH_0$-trivial resolution $\varphi \colon \tilde{X} \to X$ of singularities such that $H^0 (\tilde{X}, \Omega^{n-1}_{\tilde{X}}) \ne 0$.
\end{Prop}

\begin{proof}
Let $\mcM^{\circ}$ be the invertible subsheaf of $(\Omega_{X^{\circ}}^{n-1})^{\vee \vee}$ associated to the covering $\pi^{\circ} = \pi|_{X^{\circ}}\colon X^{\circ} \to Z^{\circ}$ and let $\mcM \subset (\Omega_X^{n-1})^{\vee \vee}$ be the pushforward of $\mcM^{\circ}$ via the open immersion $X^{\circ} \inj X$.
Note that we have
\[
\mcM^{\circ} \cong {\pi^{\circ}}^* (\omega_{Z^{\circ}} \otimes \mcL^{a_n}) \cong \mcO_{X^{\circ}} \big( d - \bsum_{i=0}^n a_i \big),
\]
and hence
\[
\mcM \cong \mcO_X \big( d - \bsum_{i=0}^n a_i \big).
\]

By Lemmas \ref{lem:admcr1} and \ref{lem:covtech}, $X$ admits a universally $\CH_0$-trivial resolution $\varphi \colon \tilde{X} \to X$ such that $\varphi^*\mcM \inj \Omega_{\tilde{X}}^{n-1}$.
We have $H^0 (X, \mcM) \ne 0$ by Condition \ref{cd1}.(5).
This shows that $H^0 (\tilde{X},\Omega_{\tilde{X}}^{n-1}) \ne 0$ and $\tilde{X}$ is not universally $\CH_0$-trivial by Lemma \ref{lem:Totaro}.
\end{proof}

\section{A supplemental result} \label{sec:supple}

In this section, as a supplement to the main theorems stated in Section \ref{sec:intro}, we give a yet another result of the failure of stable rationality of smooth weighted hypersurfaces, which will be necessary in the proof of Corollary \ref{maincor2}.

\begin{Thm} \label{thmapp}
Let $X$ be a very general smooth well formed weighted hypersurface of degree $d$ in $\mbP_{\mbC} (a_0,\dots,a_{n+1})$ which is not a linear cone.
Suppose that $n \ge 3$ and there exists $k \in \{0,\dots,n+1\}$ with the following properties:
\begin{enumerate}
\item $a_k > 1$.
\item $d > a_k (a_i - 1) (a_j-1)$ for any $0 \le i < j \le n+1$, and $d \ge a_k a_{\max}$.
\item There exists $l \ne k$ such that $d/a_k - a_l$ is divisible by $a_k$.
\item The inequality
\[
d \ge \frac{a_k}{a_k+1} a_{\Sigma}
\]
is satisfied.
\end{enumerate}
Then $X$ is not stably rational.
\end{Thm}

\begin{Rem} \label{rem:simple}
The assumptions in Theorem \ref{thmapp} are complicated.
However, when $2$ appear in the weights and $d/a_{\Pi}$ is odd, they become simple because, by choosing $a_k = 2$, the conditions (1), (2) and (3) are automatically satisfied.

In some cases, Theorem \ref{thmapp} can give results better than Theorems \ref{mainthm1}, \ref{mainthm2}, \ref{mainthm3} (see also Remark \ref{rem:ex} below):
Consider a very general weighted hypersurface $X_{2 a b} \subset \mbP_{\mbC} (1^{n-1},2,a,b)$, where $n \ge 3$, $2 < a < b$, $a, b$ are odd and coprime to each other.
Theorems \ref{mainthm2} and \ref{mainthm3} cannot be applicable to $X_{2 a b}$ and, by applying Theorem \ref{mainthm1}, we conclude the failure of stable rationality of $X_{2 a b}$ when $2 a b - a \ge n+1$.
On the other hand, we can apply Theorem \ref{thmapp} and conclude the failure of stable rationality of $X_{2 a b}$ for $3 a b - a - b \ge n+1$, which is better than the result obtained by Theorem \ref{mainthm1}.
\end{Rem}

\begin{Rem} \label{rem:ex}
We consider a very general weighted hypersurface $X_{2m} \subset \mbP_{\mbC} (1^{2m+1},2)$ of degree $2m$ for $m \ge 2$.
Failure of stable rationality of $X_{2m}$ is proved for even $m \ge 4$ by Theorem \ref{mainthm2} and for odd $m \ge 7$ by Theorem \ref{mainthm3}, and the cases $m = 2,3,5$ are not covered by the main theorems in Section \ref{sec:intro}.
By Theorem \ref{thmapp}, we can conclude that $X_{2m}$ is not stably rational for $m = 3,5$.
Moreover, $X_4 \subset \mbP (1^5,2)$ is covered by \cite{HPT}, so that $X_{2m}$ is not stably rational for any $m \ge 2$.
\end{Rem}

From now on, let $a_0,\dots,a_{n+1}, d, k$ and $p$ be are as in Theorem \ref{thmapp}.
We set $b = d/a_k$.

\begin{Lem} \label{lem:knownapp}
\emph{Theorem \ref{thmapp}} holds true if in addition one of the following condition is satisfied.
\begin{enumerate}
\item $d \ge a_{\Sigma}$.
\item $d < 3 a_{\max}$.
\end{enumerate}
\end{Lem}

\begin{proof}
Let $W = W_d \subset \mbP_{\mbC} (a_0,\dots,a_{n+1})$ be a very general smooth well formed weighted hypersurface of degree $d$.
(1) is obvious and we omit the proof.
We may assume $d < a_{\Sigma}$, that is, $W$ is Fano, in the following.

We prove (2).
Suppose that $d < 3 a_{\max}$.
Then, by Lemma \ref{lem:smwhnumerics}.(3), $d = 2 a_{\max}$ and we are in one of the cases: (i) $r = n+1$, (ii) $r = n$, (iii) $r = n-1$ and $a_n = 2$.

The case (i) does not happen since we are assuming the existence of $a_k > 1$.

Suppose that we are in case (ii).
Then we have $W = W_{2 a} \subset \mbP (1^{n+1},a)$ and $a_k = a > 1$.
The condition (4) in Theorem \ref{thmapp} is equivalent to $a \ge n-1$ which implies $I_X = n+1 - a \le 2 \le a = a_{\max}$.
Thus $W$ is not stably rational by Theorem \ref{mainthm1}.

Suppose that we are in case (iii).
Then $W = W_{2 a} \subset \mbP (1^n, 2, a)$ for some odd $a \ge 3$.
There are two possibility for the choice of $k$: either $a_k = 2$ or $a_k = a$.
If $a_k = 2$ (resp.\ $a_k = a$), then the condition (4) of Theorem \ref{thmapp} is equivalent to $2 a \ge n+2$ (resp.\ $a \ge n$), and in both cases we have $I_W = n + 2 - a \le a = a_{\max}$.
It follows that $W$ is not stably rational by Theorem \ref{mainthm1} and (2) is proved.
\end{proof}

Hence, in addition to the conditions explicitly given in Theorem \ref{thmapp}, we may assume that the following hold.

\begin{Cond} \label{cdapp}
\begin{enumerate}
\item $a_0,\dots,a_{n+1}$ are mutually coprime to each other.
\item $d$ is divisible by $a_{\Pi}$.
\item $2 r \ge n + 1$.
\item $b = d/a_k$ is coprime to $a_k$.
\item $3 a_{\max} \le d < a_{\Sigma}$.
\end{enumerate}
\end{Cond}

Note that (1), (2), (3) follows from Lemmas \ref{lem:charactsmoothwh}, and \ref{lem:smwhnumerics}, (4) follows from the condition (3) of Theorem \ref{thmapp} and (5) follows from Lemma \ref{lem:knownapp}.

In the following we choose and fix a prime number $p$ which divides $a_k$.

\begin{Rem}
By considering the variety
\[
\mcX := (y^{a_k} - f = t y - g = 0) \subset \mbP_{\mbC} (a_0,\dots,a_{n+1},b) \times \mbA^1_t,
\]
where we take $x_0,\dots,x_{n+1},y$ as homogeneous coordinates of degree $a_0,\dots,a_{n+1},b = d/a_k$, respectively, and $f, g \in \mbC [x_0,\dots,x_{n+1}]$ are very general homogeneous polynomials of degree $d, b$, respectively, we see that a very general weighed hypersurface $W$ of degree $d$ in $\mbP_{\mbC} (a_0,\dots,a_{n+1})$ degenerates to a complete intersection
\[
W' = (y^{a_k} - f = g = 0) \subset \tilde{\mbP}_{\mbC} := \mbP_{\mbC} (a_0,\dots,a_{n+1}, b).
\]
By Lemma \ref{lem:smXoapp} below, $W'$ is smooth.
We consider reduction modulo $p$ of $W'$ and set
\[
X = (y^{a_k} - f = g = 0) \subset \mbP_{\K} (a_0,\dots,a_{n+1},b),
\]
where $\K$ is an algebraically closed field of characteristic $p$ and $f, g \in \K [x_0,\dots,x_{n+1}]$ are very general homogeneous polynomials of degree $d, b$, respectively.
By the same argument as in Remark \ref{rem:spargumentThm2}, $W$ is not stably rational if there is a universally $\CH_0$-trivial resolution $\varphi \colon \tilde{X} \to X$ such that $\tilde{X}$ is not universally $\CH_0$-trivial.
\end{Rem}

\begin{Lem} \label{lem:smXoapp}
Let $K$ be an algebraically closed field.
Let $f, g \in K [x_0,\dots,x_{n+1}]$ be general homogeneous polynomials of degree $d, b = d/a_k$, respectively, and define
\[
\begin{split}
X_K &:= (y^{a_k} - f = g = 0) \subset \tilde{\mbP}_K := \mbP_{K} (a_0,\dots,a_{n+1},b), \\
\Gamma &= \bigcap_{0 \le i \le n+1, i \ne k} (x_i = 0) \subset \tilde{\mbP}_K.
\end{split}
\]
where $x_0,\dots,x_{n+1},y$ are homogeneous coordinates of degree $a_0,\dots,a_{n+1}, b$, respectively.
Then the following assertions hold.
\begin{enumerate}
\item $X_K$ is smooth along $\Gamma \cap X_K$.
\item If $\operatorname{char} (K) = 0$, then $X_K$ is smooth.
\end{enumerate}
\end{Lem}

\begin{proof}
In this proof, re-ordering the $a_i$, we assume that $k = n+1$, i.e.\ $a_k = a_{n+1}$.

We first prove (1).
We may assume that the coefficient of the degree $d$ monomial $x_{n+1}^b$ in $f$ is $1$ since $f$ is general.
Then 
\[
\Gamma \cap X_K = (x_0 = \cdots = x_n = 0) \cap X_K = (x_0 = \cdots = x_n = y^{a_{n+1}} - x_{n+1}^b = 0).
\] 
Let $\msp \in \Gamma \cap X_K$ be any point.
Then we can write $\msp = (0\!:\!\cdots\!:\!0\!:\!\alpha\!:\!\beta)$ for some non-zero $\alpha, \beta \in K$.
Since $b$ is coprime to $a_k = a_{n+1}$, we can take positive integers $\lambda,\mu$ such that $\lambda a_{n+1} - \mu b = 1$, and set $Q = x_{n+1}^{\lambda} y^{-\mu}$.
Then the open set $V = (x_{n+1} \ne 0) \cap (y \ne 0) \subset \tilde{\mbP}_K$ is isomorphic to $\mbA^{n+1}_{\tilde{x}_0,\dots,\tilde{x}_{n-1}} \times (\mbA^1_u \setminus \{o\})$, where $\tilde{x}_i = x_i/Q^{a_i}$ and $u = y^{a_{n+1}}/x_{n+1}^b$.
Note that $\msp$ corresponds to the point $(0,\dots,0,\gamma) \in \mbA^{n+1} \times (\mbA^1 \setminus \{o\})$ for some non-zero $\gamma \in K$.
Let $l \ne n+1$ be such that $b-a_l$ is divisible by $a_{n+1}$ (such an $l$ exists by the assumption of Theorem \ref{thmapp}) and write $b - a_l = m a_{n+1}$, where $m = (b-a_l)/a_{n+1}$ is a positive integer.
We may assume that the coefficient of the degree $b$ monomial $x_l x_{n+1}^m$ in $g$ is $1$ since $g$ is general.
Then we have
\[
\tilde{f} := f|_V = u^{\lambda a_{n+1}} - u^{\mu b} + \tilde{f}_1, \ 
\tilde{g} := g|_V = \tilde{x} u^{\nu} + \tilde{g}_1,
\]
where $\nu \ne 0$ is an integer, $\tilde{f}_1, \tilde{g}_1 \in (\tilde{x}_0,\dots,\tilde{x}_n)$ and $\tilde{x}_l u^{\nu}$ is the unique unique term in $\tilde{g}$ consisting only of $\tilde{x}_l$ and $u$.
We see that $X_K \cap V$ is the subvariety of $V$ defined by $\tilde{f} = \tilde{g} = 0$ and we compute
\[
\frac{\prt \tilde{f}}{\prt u} (\msp) = \frac{\prt \tilde{g}}{\prt \tilde{x}_l} (\msp) = 1,
\]
which shows that $X_K$ is smooth at any point of $\Gamma \cap X_K$.
This proves (1).

We prove (2).
We assume $\operatorname{char} (K) = 0$.
We set 
\[
Z_K := (g = 0) \subset \mbP_K := \mbP_K (a_0,\dots,a_{n+1}),
\]
and let $D_K = (f = g = 0) \subset \mbP_K$ be the divisor on $Z_K$ cut out by the equation $f = 0$.
Note that $a_0,\dots,a_{n+1}$ are mutually coprime to each other, $b = \deg g$ is divisible by $a_{\Pi}/a_{n+1}$ and $b > (a_i-1) (a_j-1)$ for any $i \ne j$.
Hence, by Lemma \ref{lem:crismcharp3}, $Z_K$ is smooth outside the point $\msq := (0\!:\!\cdots\!:\!0\!:\!1) \in \mbP_K$.
We see that $D_K$ is a general member of the base point free linear system $|\mcO_{Z_K} (d)|$.
By Bertini theorem, $D_K$ is smooth.
Let $\pi \colon X_K \to Z_K$ be the natural morphism.
Then, since $Z_K$ is smooth outside $\msq$ and $X \setminus \pi^{-1} (\msq) \to Z_K \setminus \{\msq\}$ is a cyclic covering branched along the smooth divisor $D_K$, we conclude that $X_K \setminus \pi^{-1} (\msq)$ is smooth.
We have $\pi^{-1} (\msq) = \Gamma \cap X_K$.
By (1), $X_K$ is smooth at any point of $\pi^{-1} (\msq)$, hence $X_K$ is smooth.
\end{proof}

In the following, we work over an algebraically closed field $\K$ of characteristic $p$, where we recall that $p$ divides $a_k$.
Let $f, g \in \K [x_0,\dots,x_{n+1}]$ be general homogeneous polynomials of degree $d, b = d/a_k$, respectively.
We set
\[
\begin{split}
X &= (y^{a_k} - f = g = 0) \subset \mbP_{\K} (a_0,\dots,a_{n+1},b), \\
Z &= (g = 0) \subset \mbP := \mbP_{\K} (a_0,\dots,a_{n+1}),
\end{split}
\]
and let $\pi \colon X \to Z$ be the natural morphism.
By the above argument Theorem \ref{thmapp} follows if we show the existence of a universally $\CH_0$-trivial resolution of singularities $\varphi \colon \tilde{X} \to X$ such that $\tilde{X}$ is not universally $\CH_0$-trivial.

For $i = 0,\dots,n+1$, we define $\msq_i = (0\!:\!\cdots\!:\!1\!:\!\cdots\!:\!0) \in \mbP$, where the unique $1$ is in the $(i+1)$st position.
Since $b$ is not divisible by $a_k$, we have $\msq_k \in Z$.
We set $Z^{\circ} = Z \setminus \{\msq_k\}$ and $X^{\circ} = \pi^{-1} (Z^{\circ})$.

\begin{Lem}
$Z^{\circ}$ is smooth, and $X$ is smooth along $X \setminus X^{\circ}$.
\end{Lem}

\begin{proof}
The first and second assertions follows from Lemmas \ref{lem:crismcharp3} and \ref{lem:smXoapp}.(1).
\end{proof}

We set $\mcL = \mcO_{Z^{\circ}} (b)$, which is an invertible sheaf.
We can view $f$ (or more precisely $f|_{Z^{\circ}}$) as a global section of $\mcL^{a_k} = \mcO_{Z^{\circ}} (d)$.
The restriction $\pi^{\circ} = \pi|_{X^{\circ}} \colon X^{\circ} \to Z^{\circ}$ is the covering obtained by taking the $a_k$th roots of $f \in H^0 (Z^{\circ}, \mcL^{a_k})$.
In the following, re-ordering the $x_i$, we assume $a_0 = \cdots = a_r = 1$ and  $a_i > 1$ for any $i > r$.
We set $\Delta = (x_0 = \cdots = x_r = 0) \subset \mbP$, $\Delta^{\circ} = \Delta \setminus \{\msq_k\}$ and $\Delta_Z = \Delta \cap Z$, $\Delta^{\circ}_Z = \Delta^{\circ} \cap Z$.

\begin{Lem} \label{lem:admcrapp}
A general $f \in H^0 (Z^{\circ}, \mcL^p)$ has only admissible critical points on $Z^{\circ}$.
\end{Lem}

\begin{proof}
For $r < i \ne j \le n+1$, we set $U_{i,j} = (x_i \ne 0) \cap (x_j \ne 0) \subset \mbP$.
We have
\[
\Delta^{\circ}_Z = \bigcup_{r < i < j \le n+1, i,j \ne k} \Delta_Z \cap U_{i,j}.
\]
Since $d \ge a_i a_j$ for any $i, j \ne k$, we can apply Lemma \ref{lem:restP}.(2) (cf.\ Remark \ref{rem:restZ}) and the image of the restriction map
\[
H^0 (Z, \mcO_Z (d)) = H^0 (Z^{\circ}, \mcL^{a_k}) \to \mcL^{a_k} \otimes (\mcO_{Z^{\circ}}/\mfm_{\msp}^2)
\]
is of dimension at least $r$, as a $\K$-vector space, for any $\msp \in \Delta^{\circ}_Z$.
Since $\dim \Delta^{\circ}_Z = n - (r+1)$ and $2 r \le n + 1$, we can conclude by counting dimensions that a general $f \in H^0 (Z^{\circ}, \mcL^p)$ does not have a critical point along $\Delta^{\circ}_Z$.

The rest of the proof is completely the same as that of Lemma \ref{lem:admcr2}, where all we need is the condition $d \ge 3 a_{\max}$.
\end{proof}

\begin{Prop}
The variety $X$ admits a universally $\CH_0$-trivial resolution $\varphi \colon \tilde{X} \to X$ of singularities such that $\tilde{X}$ is not universally $\CH_0$-trivial.
\end{Prop}

\begin{proof}
Let $\mcM^{\circ}$ be the invertible subsheaf of $(\Omega_{X^{\circ}}^{n-1})^{\vee \vee}$ associated to the covering $\pi^{\circ} \colon X^{\circ} \to Z^{\circ}$ and $\mcM$ the pushforward of $\mcM^{\circ}$ via the open immersion $X^{\circ} \inj X$.
We have
\[
\begin{split}
\mcM^{\circ} &\cong {\pi^{\circ}}^* (\omega_{Z^{\circ}} \otimes \mcL^{a_k}) \cong \mcO_{X^{\circ}} (b - a_{\Sigma} + d), \\
\mcM &\cong \mcO_X (b - a_{\Sigma} + d).
\end{split}
\]
By Lemmas \ref{lem:admcrapp} and \ref{lem:covtech}, $X$ admits a universally $\CH_0$-trivial resolution $\varphi \colon \tilde{X} \to X$ such that $\varphi^*\mcM \inj \Omega_{\tilde{X}}^{n-1}$.
We see that $H^0 (X, \mcM) \ne 0$ since
\[
b - a_{\Sigma} + d = \frac{a_k+1}{a_k} d - a_{\Sigma} \ge 0.
\]
This shows $H^0 (\tilde{X}, \Omega_{\tilde{X}}^{n-1}) \ne 0$, hence $\tilde{X}$ is not universally $\CH_0$-trivial by Lemma \ref{lem:Totaro}.
\end{proof}

\section{Proof of Corollaries} \label{sec:pfcor}

\begin{proof}[Proof of \emph{Corollary \ref{maincor1}}]
Let $X$ be a very general smooth well formed weighted hypersurface of degree $d$ and index $I$ in $\mbP_{\mbC} (a_0,\dots,a_{n+1})$.
Note that $d = a_{\Sigma} - I$ and we have $a_{\Sigma} \ge n+2$.
If $I \le a_{\max}$, then $X$ is not stably rational.
Hence, in the following we assume that $a_i < I$ for any $i$.
Since $a_0,\dots,a_{n+1}$ are mutually coprime to each other, this in particular implies $a_{\Pi} \le (I-1)!$.

Suppose that $d/a_{\Pi} \ge 2$ is even.
If $n \ge 3 I - 2$, then
\[
d - \frac{2}{3} a_{\Sigma} = \frac{1}{3} a_{\Sigma} - I = \frac{1}{3} (n+2) - I \ge 0.
\]
Thus, by Theorem \ref{mainthm2}., $X$ is not stably rational.

Suppose that $d/a_{\Pi} = 3$.
If $n \ge 4 I - 2$, then by the similar argument as above, we have $d \ge \frac{3}{4} a_{\Sigma}$ and, by Theorem \ref{mainthm2}, $X$ is not stably rational.

Suppose that $d/a_{\Pi} \ge 5$ is odd.
If $n \ge 3 I + 3 (I-1)! - 2$, then
\[
d - a_{\Pi} - \frac{2}{3} a_{\Sigma} = \frac{1}{3} a_{\Sigma} - a_{\Pi} - I \ge \frac{1}{3} (n+2) - (I-1)! - I \ge 0.
\]
Hence, by Theorem \ref{mainthm3}, $X$ is not stably rational.

Finally, suppose that $d/a_{\Pi} = 1$.
Let $(b_1,\dots,b_m)$ be a tuple of mutually coprime integers such that $2 \le b_1 < \cdots < b_m < I$ and $b_{\Sigma} - b_{\Pi} \le I$, where $b_{\Sigma} = b_1 + \cdots + b_m$ and $b_{\Pi} = b_1 \cdots b_m$.
Then a general smooth well formed weighted hypersurface of degree $b_{\Pi}$ in $\mbP (1^{I+b_{\Pi} - b_{\Sigma}},b_1,\dots,b_m)$ is of index $I$.
Conversely, if we are given a smooth well formed hypersurface of degree $d = a_{\Pi}$ in $\mbP (a_0,\dots,a_{n+1})$, then $(a_{r+1},\dots,a_{n+1})$, where $r$ is defined in such a way that $a_i > 1$ if and only if $i > r$, is a tuple of mutually coprime integers satisfying the above conditions.
It is easy to see that a set of tuples $(b_1,\dots,b_m)$ such that $2 \le b_1 < \cdots < b_m < I$ is finite.
This means that there are only finitely many combinations of the weights $(a_0,\dots,a_{n+1})$ such that a weighted hypersurfaces of degree $a_{\Pi}$ in $\mbP (a_0,\dots,a_{n+1})$ is smooth, well formed and of index $I$.
Thus there exists a number $N'_I$ depending only on $I$ such that if $n \ge N'_I$, then $X$ does not exists.

We set 
\[
N_I = \max \{ 3I-2,4I-2,3 I + 3 (I-1)!-2, N'_I \}.
\]
Then the assertion in the corollary holds for this $N_I$.
\end{proof}

\begin{proof}[Proof of \emph{Corollary \ref{maincor2}}]
Let $X = X_d \subset \mbP_{\mbC} (a_0,\dots,a_{n+1})$ be a very general smooth well formed hypersurface of degree $d$.
(1) follows immediately from Theorem \ref{mainthm1}.

Suppose that $I_X = 2$.
Then $I_X > a_{\max}$ if and only if $a_0 = \cdots = a_{n+1} = 1$.
In this case $X$ is a degree $n$ hypersurface in $\mbP^{n+1}$ which is not stably rational except when $n = 3$ by \cite{Totaro}.
This proves (2).

Suppose that $I_X = 3$.
Then $I_X > a_{\max}$ if and only if either $X$ is a hypersurface of degree $n-1$ in $\mbP^{n+1}$ or $X$ is a weighted hypersurface of degree $2 m$ in $\mbP (1^{2m+1},2)$ for $m \ge 2$.
If $X$ is a hypersurface of degree $n-1$ in $\mbP^{n+1}$, then $X$ is not stably rational except when $n = 3,4,5,6,8$ by \cite{Totaro}.
We see from Remark \ref{rem:ex} that $X = X_{2m} \subset \mbP (1^{2m+1},2)$ is not stably rational for any $m \ge 2$.
This proves (3).
\end{proof}

\end{document}